\documentclass[12pt]{amsart}
\usepackage{txfonts}
\usepackage{amscd,amssymb,amsmath,amsxtra,relsize}
 \usepackage[foot]{amsaddr}
\usepackage[pdftex,colorlinks,plainpages,backref,urlcolor=blue]{hyperref}
\usepackage{url}
\usepackage{breakurl}
\usepackage{array}
\usepackage{enumitem}
\usepackage{tikz}
\usetikzlibrary{cd}
\usepackage{longtable}
\usepackage{tabu}
\usepackage{shuffle}
\usepackage{verbatim}

\newtheorem{theorem}{Theorem}[section]
\newtheorem{lemma}[theorem]{Lemma}
\newtheorem{corollary}[theorem]{Corollary}
\newtheorem{proposition}[theorem]{Proposition}

\theoremstyle{definition}
\newtheorem{definition}[theorem]{Definition}
\newtheorem{example}[theorem]{Example}
\newtheorem{remark}[theorem]{Remark}
\newtheorem{question}[theorem]{Question}

\newtheorem*{ack}{Acknowledgment}

\DeclareMathOperator{\T}{\mathsf{T}}

\DeclareMathOperator{\Hom}{Hom}

\DeclareMathOperator{\Int}{Int}
\DeclareMathOperator{\Sym}{Sym}
\DeclareMathOperator{\Frac}{Frac}
\DeclareMathOperator{\Bin}{Bin}

\DeclareMathOperator{\dash}{-}

\DeclareMathOperator{\supp}{supp}

\DeclareMathSymbol{\Gamma}{\mathalpha}{operators}{0}

\newcommand{\N}{\mathbb{N}}
\newcommand{\Z}{\mathbb{Z}}
\newcommand{\Q}{\mathbb{Q}}
\newcommand{\R}{\mathbb{R}}

\newcommand{\fm}{\mathfrak{m}}

\newcommand{\bx}{\mathbf{x}}
\newcommand{\bz}{\mathbf{0}}
\newcommand{\ot}{\otimes}

\newcommand{\z}{\zeta}

\newcommand{\la}{\langle}
\newcommand{\ra}{\rangle}

\newcommand{\ba}{\mathbf{a}}

\DeclareMathAlphabet{\pazocal}{OMS}{zplm}{m}{n}

\newcommand{\surj}{\twoheadrightarrow}

\newcommand{\cupd}{\cup_1{\dash}d}

\makeatletter
\newsavebox{\@brx}
\newcommand{\llangle}[1][]{\savebox{\@brx}{\(\m@th{#1\langle}\)}%
  \mathopen{\copy\@brx\kern-0.5\wd\@brx\usebox{\@brx}}}
\newcommand{\rrangle}[1][]{\savebox{\@brx}{\(\m@th{#1\rangle}\)}%
  \mathclose{\copy\@brx\kern-0.5\wd\@brx\usebox{\@brx}}}
\makeatother

\newcommand{\arxiv}[1]
{\texttt{\href{http://arxiv.org/abs/#1}{arXiv:#1}}}
\newcommand{\arxi}[1]
{\texttt{\href{http://arxiv.org/abs/#1}{arxiv:}}
\texttt{\href{http://arxiv.org/abs/#1}{#1}}}

\newcommand{\abs}[1]{\left| #1 \right|}

\numberwithin{table}{section}
\numberwithin{equation}{section}
\def\norms#1#2{\left\| #1 \right\|_{\lower 1ex\hbox{$\scriptstyle #2 $}}}

\makeatletter
\def\namedlabel#1#2{\begingroup
    #2%
    \def\@currentlabel{#2}%
    \phantomsection\label{#1}.\endgroup
}
\makeatother
\makeatletter
\def\@tocline#1#2#3#4#5#6#7{\relax
  \ifnum #1>\c@tocdepth 
  \else
    \par \addpenalty\@secpenalty\addvspace{#2}%
    \begingroup \hyphenpenalty\@M
    \@ifempty{#4}{%
      \@tempdima\csname r@tocindent\number#1\endcsname\relax
    }{%
      \@tempdima#4\relax
    }%
    \parindent\z@ \leftskip#3\relax \advance\leftskip\@tempdima\relax
    \rightskip\@pnumwidth plus4em \parfillskip-\@pnumwidth
    #5\leavevmode\hskip-\@tempdima
      \ifcase #1
       \or\or \hskip 1em \or \hskip 2em \else \hskip 3em \fi%
      #6\nobreak\relax
    \dotfill\hbox to\@pnumwidth{\@tocpagenum{#7}}\par
    \nobreak
    \endgroup
  \fi}
\makeatother

\definecolor{dkgreen}{RGB}{0,100,0}
\definecolor{dkbrown}{RGB}{139,69,19}

\makeatletter
\@namedef{subjclassname@2020}{%
  \textup{2020} Mathematics Subject Classification}
\makeatother

\begin{document}
 
\date{December 6, 2021}
    
\title[Cup-one algebras, binomial operations, and Massey products]%
{Differential graded algebras, Steenrod cup-one products, 
binomial operations, and Massey products}

\author{Richard~D.~Porter$^1$}
\author{Alexander~I.~Suciu$^{1,2}$}
\address{$^1$Department of Mathematics,
Northeastern University,
Boston, MA 02115, USA}
\email{\href{mailto:r.porter@northeastern.edu}{r.porter@northeastern.edu}}

\email{\href{mailto:a.suciu@northeastern.edu}{a.suciu@northeastern.edu}}
\thanks{$^2$Supported in part by the Simons Foundation Collaboration Grants 
for Mathematicians \#354156 and \#693825}

\subjclass[2020]{Primary
16E45, 
55N45. 
Secondary
13F20, 
46L87,  
55P62, 
55S05, 
55S10, 
55S30, 
55U10. 
}

\keywords{Differential graded algebras, cochain algebras, 
Steenrod cup-$i$ products, 
binomial rings, rings of integer-valued polynomials, 
non-commutative differential forms, Massey products}

\dedicatory{In memory of \c{S}tefan Papadima}

\begin{abstract}
Motivated by the construction of Steenrod cup-$i$ products in 
the singular cochain algebra of a space and in the 
algebra of non-commutative differential forms, we define a
category of binomial cup-one differential graded algebras 
over the integers and over prime fields of positive characteristic. 
The Steenrod and Hirsch identities bind the cup-product, the 
cup-one product, and the differential in a package that we 
further enhance with a binomial ring structure arising from 
a ring of integer-valued rational polynomials. This structure 
allows us to define the free binomial cup-one differential
graded algebra generated by 
a set and derive its basic properties. It also provides the 
context for defining restricted triple Massey products, which 
have a smaller indeterminacy than the classical ones, and
hence, give stronger homotopy type invariants.
\end{abstract}

\maketitle

\setcounter{tocdepth}{1}
\tableofcontents

\section{Introduction}
\label{sect:intro}

\subsection{Algebraic models for spaces}
\label{intro:models}

From its beginnings, algebraic topology has relied on algebraic models 
to capture topological properties of spaces. Some of the earliest such 
models are the cochain algebra of a simplicial complex and the de Rham 
algebra of smooth differential forms on a smooth manifold. The first 
construction was extended to cellular cochains on a CW-complex 
and to singular cochains on an arbitrary topological space. The 
latter construction evolved into Sullivan's algebra of 
piecewise polynomial forms on a topological space, 
which provides a natural setting for rational homotopy theory, 
\cite{Sullivan}. 

An important distinction in this context is between the commutative 
differential graded algebra (cdga) models such as the de Rham and 
Sullivan algebras and the non-commu\-tative differential graded algebra 
(dga) models such as the aforementioned cochain algebras, as well 
as the algebras of non-commutative differential forms. Although 
in some ways more difficult to handle, these dga models 
contain extra information, exemplified by the Steenrod 
cup-$i$ products, which provide explicit witnesses to the 
non-commutativity of the usual cup-products and lead to
the Steenrod cohomology operations.

We revisit here some of these classical ideas, many of which 
go back to the work of Steenrod \cite{Steenrod} and Hirsch \cite{Hirsch}, 
and define several categories of differential graded algebras (over the 
integers and over prime fields of positive characteristic) with some 
extra structure, coming from either the cup-one products, 
or from a binomial ring structure, or both, bound together 
by suitable compatibility conditions. This structure 
allows us to define the free binomial cup-one differential
graded algebra generated by a set, which will be developed 
in \cite{Porter-Suciu-2021-1, Porter-Suciu-2021-2}  into a theory 
of $1$-minimal models over the integers and over $\Z_p$.  This 
framework also provides the context for defining restricted triple 
Massey products, which is a particular case of a more general 
construction that will be developed in \cite{Porter-Suciu-2021-3}. 

\subsection{Differential graded algebras and $\cup_1$-products}
\label{intro:dga-cup-one}
We now outline the main constructions and results of this paper. 
We start in section \ref{sect:dga} by reviewing some basic notions 
regarding graded algebras and differential graded algebras, including 
Massey's definition of higher order products \cite{Massey-1958}. 
In section \ref{sect:steenrod-cup} we enhance (differential) graded 
algebras with an additional structure, working by analogy with 
Steenrod's cup-one products in cochain algebras.

We define a {\em graded algebra with cup-one products}\/ over a coefficient 
ring $R$ as a graded $R$-algebra $A$ with cup-one product maps, 
$\cup_1\colon A^p \otimes_R A^1 \to A^{p}$ for $p=1,2$ such that
the {\em left Hirsch identity},
\begin{equation}
\label{eq:intro-hirsch-1c}
(a \cup b) \cup_1 c = a \cup (b\cup_1 c) + (a \cup_1 c)\cup b \, ,
\end{equation} 
is satisfied for all $a,b,c\in A^1$, and such that 
the cup-one map $\cup_1\colon A^1 \otimes_R A^1 \to A^1$ gives 
the $R$-module $R \oplus A^1$ the structure of a commutative ring.

Suppose now that $(A,d)$ is a dga such that the underlying algebra $A$ 
satisfies the above two conditions, and also comes endowed with a map
$\cup_1 \colon A^1 \otimes A^2 \to A^2$. 
We then say that $(A,d)$ is a 
 {\em differential graded algebra with $\cup_1$-products}\/ if the following 
{\em Steenrod identity}\/ is satisfied for all $a,b\in A^1$:
\begin{equation}
\label{eq:intro-steenrod-1c}
d(a \cup_1 b)  = -a\cup b- b\cup a  + da \cup_1 b - a\cup_1 db \, .
\end{equation}

Note that by the left Hirsch identity if $da$ is a sum of cup products
then the first three terms in equation \eqref{eq:intro-steenrod-1c}
belong to $D^2(A)$, the set of elements in $A^2$ that can be 
written as sums of cup products of elements in $A^1$.
This raises the question of whether there are
conditions under which the term $a\cup_1 db$  can also
be written as a sum of cup products of elements in $A^1$.
To this end, we assume there is an operation on $D^2(A)$ 
satisfying
$(u \cup v)\circ (w \cup z)=(u \cup_1 w) \cup (v \cup_1 z)$, 
and consider the identities 
\begin{align}
\label{eq:intro-right-hirsch}
a \cup_1 (b\cup c) &= da \circ (b \cup c) - (b\cup c) \cup_1 a,
\\
\label{eq:intro-c1d}
d(a \cup_1 b) &= - a \cup b - b \cup a
+ da \cup_1 b + db \cup_1 a - da \circ db,
\end{align}
for $a,b,c \in A^1$ and $da,db \in D^2(A)$. We call the first one 
the {\em right Hirsch identity} and the second one the 
{\em $\cupd$ formula}. If \eqref{eq:intro-steenrod-1c} and 
\eqref{eq:intro-right-hirsch} hold, then it follows that
\eqref{eq:intro-c1d} is also satisfied. 
A differential graded $R$-algebra $(A,d)$ with cup-one 
products and differential that satisfies the 
$\cupd$ formula is called a {\em $\cup_1$-differential 
graded algebra}.

We show in sections \ref{subsec:cochains-cup1d} and 
\ref{subsec:noncomm-cup1d}
that the right Hirsch identity, and hence the $\cupd$ formula,
holds in cochain algebras and in
algebras of non-commutative differential forms, respectively.
The property that the $\cupd$ formula involves only cup products
of elements in $A^1$ is a critical component in the proofs
in sections \ref{sect:free-bin-dga} and \ref{sect:zp-binomial}
of the properties of free $\cup_1$-differential graded algebras.

\subsection{Cochain algebras and non-commutative differential forms}
\label{intro:cochain-nc}
We have two basic motivating examples for this definition. 
The first one, discussed in section \ref{sect:cochain}, is 
the singular cochain algebra, $A=C^*(X;R)$, 
of a topological space $X$, with coefficients in a commutative 
ring $R$.  In this case, one can take the map $\cup_1$ to be the 
cup-one product defined by Steenrod in \cite{Steenrod}, with the  
Steenrod and Hirsch identities established in \cite{Steenrod} and \cite{Hirsch},
respectively. 

The second motivating example, discussed in section \ref{sect:non-comm},  
is provided by the algebra of non-commutative 
differential forms, $\Omega^*(A;R)$, on an $R$-algebra $A$, 
as constructed by Karoubi in \cite{Karoubi-tams1995}.  
In this context, the Steenrod identities were  
established by Battikh \cite{Battikh-2007} and the Hirsch identities by 
Abbassi \cite{Abbassi-2013}. 
As we shall see in \cite{Porter-Suciu-2021-1}, sequences of non-commutative 
Hirsch extensions also satisfy the aforementioned axioms.

\subsection{Binomial cup-one dgas}
\label{intro:bin-co-one}
Inspired by the work of Ekedahl \cite{Ekedahl-2002}, Elliott \cite{Elliott-2006}, 
and others we bring in section \ref{sect:bin-cup1}  the notion of 
binomial ring and tie it up to the previous examples  
in order to define binomial cup-one algebras and dgas. 

A commutative ring $A$ is a {\em binomial ring}\/ if $A$ is torsion-free 
as a $\Z$-module, and the elements 
$\binom{a}{n} \coloneqq a(a-1)\cdots (a-n+1)/n!$ 
lie in $A$ for every $a\in A$ and every $n>0$, thereby defining 
maps $\zeta_n\colon A\to A$, $a\mapsto \binom{a}{n}$.
Combining the notions of cup-one (differential) graded algebras 
and binomial rings we arrive at the notion of a binomial cup-one 
(differential) graded algebra. We define binomial $\cup_1$-dgas as 
graded algebras with cup-one products endowed with a differential 
satisfying the Leibniz and $\cupd$ formulas.

Our goal here is to derive a formula for the differential of $\zeta_n(a)$ for 
$a$ a cocycle in $A^1$ and $n\ge 1$, in the setting where $A$ is a 
binomial $\cup_1$-dga.  We achieve this in Theorem \ref{thm:cup-zeta}, 
where we show that 
\begin{equation}
\label{eq:intro-dcup}
d\zeta_n (a) = - \sum_{k = 1}^{n-1}\zeta_{k}(a) \cup \zeta_{n-k}(a) \, .
\end{equation}
In fact, we give in \eqref{eq:d-zeta} a formula for the differential of 
an arbitrary product of the form $\zeta_{i_1}(a_1)\cdots \zeta_{i_n}(a_n)$ 
with $a_k\in Z^1(A)$ and $i_k\ge 0$. 

\subsection{Free binomial graded algebras}
\label{intro:free}
A basic example of binomial cup-one algebras is provided by the free 
binomial graded algebras constructed in section \ref{sect:free-bin-dga}. 
Let $\Int(\Z^{X})$ be the ring of integrally-valued polynomials with variables 
from $X$ and with rational coefficients. This is a binomial ring, freely 
generated as a $\Z$-module by products of polynomials of the form 
$\binom{x}{n}$, with $x\in X$ and $n\in \N$.  We define the 
{\em free binomial graded algebra}\/ on $X$, denoted $\T(X)$, 
as the tensor algebra on the ideal of all polynomials in $\Int(\Z^X)$ 
without constant term.

We develop here some of the basic properties of these algebras. 
For instance, we show in Theorem \ref{thm:cup1-t} that $(\T(X),d)$, 
with differential $d$ defined by $d(x) = 0$ for all $x \in X$, together 
with the $\cupd$ formula and the graded Leibniz rule, is a binomial 
$\cup_1$-dga. Moreover, if $A$ is an any binomial $\cup_1$-dga, 
we show in Corollary \ref{cor:freebindga} that 
there is a bijection between binomial $\cup_1$-dga maps 
$\T(X)\to A$ and set maps $X\to Z^1(A)$.

Minimal models are then instances of binomial $\cup_1$-dga algebras 
that are free as graded algebras with cup-one products. 
This definition captures the properties of cup and cup-one 
products that we will use in \cite{Porter-Suciu-2021-1} to 
construct $1$-minimal models over the integers. 

\subsection{Binomial cup-one dgas over $\Z_p$}
\label{intro:zp}

In section \ref{sect:zp-binomial} we extend the results in the
previous two sections to binomial rings and $\cup_1$-dgas 
over $\Z_p$, the prime field of characteristic $p>0$.
We say that  a commutative $\Z_p$-algebra $A$ is a 
{\em $\Z_p$-binomial algebra}\/ if 
$a(a-1)\cdots(a-n+1)= 0$ for all integers $n \ge p$ and all $a \in A$. 
The binomial operations $\z_n(a)$ are defined as for algebras over $\Z$, 
but now only for $1 \le n \le p-1$.  

With this definition in hand, we proceed as before and introduce 
the notions of {\em $\Z_p$-binomial\/ $\cup_1$-differential graded algebra}\/ 
and {\em free $\Z_p$-binomial graded algebra}. If $(A,d)$ is a 
$\Z_p$-binomial $\cup_1$-dga, we show in 
Theorem \ref{thm:cup-zeta-p} that 
\begin{equation}
\label{eq:intro-dcup-p}
d\z_n(a) = - \sum_{k=1}^{n-1}\z_k(a) \cup \z_{n-k}(a)
\end{equation}
for all $a\in A^1$ with $da =0$ and for $2\le n \le p-1$. These notions 
are further developed in \cite{Porter-Suciu-2021-1} into a theory 
of $1$-minimal models over $\Z_p$.

\subsection{Ordinary and restricted triple Massey products}
\label{intro:massey}

We use the constructions above to analyze in some detail special 
cases of the ordinary Massey products, and also define restricted triple 
Massey products which have smaller indeterminacy than the
usual products. For instance, if $A$ is a binomial $\cup_1$-dga 
over $\Z$, it follows from the existence of binomials, 
$\z_i(a)$, satisfying \eqref{eq:intro-dcup} that the 
$n$-fold Massey products $\la u, \ldots, u\ra$ 
are defined and contain $0$, for all $u\in H^1(A)$ and $n\ge 3$.

Given a cocycle $a \in C^1(X;\Z_3)$,  formula
\eqref{eq:intro-dcup-p} is used in the proof of
Proposition \ref{prop:mtp-3} to show that the triple product 
$\langle [a], [a], [a] \rangle \in H^2(X;\Z_3)$ may be represented 
by the cocycle $- a\cup \z_2(a) - \z_2(a) \cup a$. We show that 
this triple product contains the negative of the mod $3$
Bockstein applied to $[a]$.  More generally, we 
show in Theorem \ref{thm:mod-p-bock}
that $p$-fold Massey products 
$\langle u, \dots , u \rangle \in H^2(X;\Z_p)$, for odd primes $p$, 
are defined and contain the negative of the 
mod $p$ Bockstein applied
to $u$. Since the indeterminacy of these $p$-fold
products in the cohomology of the Eilenberg--MacLane
space $K(\Z_p,1)$ is zero, it follows in general that the
$p$-fold product $\la u, \ldots , u\ra$ does not necessarily
contain zero.

The category of binomial $\cup_1$-dgas provides a natural framework for 
defining generalized Massey products. In section \ref{sect:massey} 
we consider a special case of such cohomology operations, 
which we call {\em restricted triple Massey products}. 
Furthermore, we relate these products and the binomial 
operations $\z_n$ in a binomial $\cup_1$-dga. 

More precisely, if $A$ is a binomial $\cup_1$-dga and  
$u_1, u_2$ are elements in $H^1(A)$ with $u_1 \cup u_2 = 0$, 
we define the restricted Massey product 
$\langle u_1, u_1, u_2 \rangle_r$ to be the
subset of $H^2(A)$ consisting of elements
$[a_1 \cup a_{12} - \z_2(a_1)\cup a_2]$, where
each $a_i$ is a cocycle with $[a_i]=u_i$ and
$da_{12} = a_1 \cup a_2$. We show that these restricted 
triple Massey products are homotopy invariants with generally 
smaller indeterminacy than the classical Massey products. 

We conclude with a construction of a family of spaces whose 
homotopy types cannot be distinguished by the usual cup products 
and triple Massey products, yet can be distinguished using our 
restricted Massey products. The theory of generalized Massey
products continues the program initiated in
\cite{Porter-Suciu-2020} and is developed more 
fully in \cite{Porter-Suciu-2021-2, Porter-Suciu-2021-3}, 
along with further applications.

\section{Differential graded algebras}
\label{sect:dga}

We start with some basic definitions.
Throughout this section we work over a fixed coefficient ring $R$, 
assumed to be commutative and with unit $1=1_R$.  

\subsection{Graded algebras}
\label{subsec:ga}
For us, an {\em $R$-algebra}\/ is an associative, unital algebra $A$ 
over a ring $R$. That is to say, $A$ is a ring (with multiplicative 
identity $1_A$) which is also an $R$-module, such that the ring and 
module additions coincide, and scalar multiplication satisfies 
$r(ab)=(ra)b=a(rb)$ for all $r\in R$ and $a,b\in A$.  Consequently, 
the ring multiplication map, $A\times A\to A$, is $R$-bilinear, and 
thus factors through a map $A\otimes_{R} A\to A$.  Moreover, $A$ 
comes endowed with a structure map, $R\to A$, given by 
$r\mapsto r\cdot 1_A$. 
We will assume throughout that the structure map is injective, 
so that $R$ may be viewed as a subring of $A$.

A {\em graded algebra}\/ over $R$ is an $R$-algebra $A$ 
such that the underlying $R$-module is a direct sum of 
$R$-modules, $A=\bigoplus_{i\ge 0} A^i$, and such that the
multiplication map $A\otimes_{R} A\to A$ sends 
$A^i \otimes_{R} A^j$ to $A^{i+j}$.

Observe that $A^0$ is a subring of $A$; thus, $1_{A^0}$  
coincides with $1_{A}$, and so $1=1_A$ is a homogeneous 
element of degree $0$. 
Consequently, $R$ may be viewed as a subring of $A^0$, and the 
graded pieces $A^i$ may be viewed as $A^0$-modules. 
We say that $A$ is {\em connected}\/ if the structure map 
$R\to A$ maps $R$ isomorphically to $A^0$. 

We refer to the multiplication maps $\cup\colon A^i\otimes_{R} A^j\to A^{i+j}$, 
given by $\cup(a\otimes b)=a b$ as the {\em cup-product maps}.
We may refer to the elements of $A^{i}$ as $i$-cochains, and, when 
there are other products involved, we will sometimes 
write $a\cup b$ for the product $ab$. 
We say that $A$ is {\em graded commutative} (for short, $A$ is a cga) if 
$ab=(-1)^{\abs{a}\abs{b}} ba$ for all homogeneous elements $a, b\in A$ 
where $\abs{a}$, $\abs{b}$ denotes the degrees of $a$ and $b$; respectively.
Some basic examples to keep in mind are:
\begin{enumerate}
\item \label{alg1} The exterior algebra $\bigwedge V$, where $V$ is 
a free $R$-module with generators in odd degrees.
\item \label{alg2} The symmetric algebra $\Sym(V)$, where $V$ is 
a free $R$-module with generators in even degrees.
\item \label{alg3} The tensor algebra $T(V)$, where $V$ is 
a free $R$-module of rank at least $2$.
\end{enumerate}
All three examples are graded algebras; the first two are graded-commutative, 
while the third one is not.

A morphism of graded algebras is a map $\varphi\colon A\to B$ of 
$R$-algebras preserving degrees, that is, $\varphi(A^i)\subset B^i$ 
for all $i$. The set of all such morphisms is denoted by $\Hom(A,B)$.

\subsection{Differential graded algebras}
\label{subsec:dga}

A {\em differential graded algebra}\/ over a ring $R$  (for short, a dga) 
is a graded $R$-algebra $A=\bigoplus_{i\ge 0} A^i$  endowed with a 
degree $1$ map, $d\colon A\to A$, such that $d^2=0$ and the following 
``graded Leibniz rule" is satisfied, 
\begin{equation}
\label{eq:derivation}
d(a \cup b) = da\cup b + (-1)^{\abs{a}} a\cup db
\end{equation}
for all homogenous elements $a, b\in A$. 

We denote by $[a]\in H^i(A)$ the cohomology class of a cocycle 
$a\in Z^i(A)$.  The graded $R$-module $H^*(A)$ inherits from 
$A$ the structure of a graded algebra, with multiplication given 
by $[a]\cdot [b]=[ab]$. We say that a dga $A$ is {\em graded commutative} 
(for short, $A$ is a cdga) if the underlying graded algebra is a cga. 
Clearly, if $A$ is a cdga, the $H^*(A)$ is a cga.

A morphism of differential graded $R$-algebras is an $R$-linear 
map, $\varphi\colon A \to B$, between two dgas which preserves 
the grading and commutes with the respective differentials and products;  
we denote the set of such morphisms by $\Hom( A, B )$.
The induced map in cohomology, $\varphi^*\colon H^*(A)\to H^*(B)$, is a 
morphism of graded $R$-algebras.   
The map $\varphi$  is called a \emph{$q$-quasi-isomor\-phism}\/ (for some 
$q\ge 1$) if the induced homomorphism, $\varphi^*\colon H^i (A) \to H^i(B)$,  
is an isomorphism for $i \le q$ and a monomorphism for $i=q+1$. 
Two dgas are called \emph{$q$-equivalent}\/ if there is a zig-zag 
of $q$-quasi-isomorphisms connecting one to the other.
A dga $(A,d)$ is said to be {\em $q$-formal}\/ if it is $q$-equivalent 
to its cohomology algebra, $H^*(A)$, endowed with the $0$ differential.   

\subsection{Massey products}
\label{subsec:massey-prod}
A well-known obstruction to formality is provided by the higher-order 
Massey products, introduced by W.S.~Massey in \cite{Massey-1958,May-1969}.  
For simplicity, we focus here on Massey triple products of cohomology 
classes in degree $1$, following the approach from \cite{Fenn-Sjerve-1984,Porter}. 

Let $(A,d)$ be an $R$-dga, and let $u_1, u_2, u_3\in H^1(A)$. 
In general, the Massey triple product $\langle u_1, u_2, u_3 \rangle$ 
is defined if $u_1u_2 = u_2 u_3 = 0$. If the product is defined, then there
are cochains $a_1, a_2, a_3, a_{1,2}, a_{2,3}\in A^1$ with the properties
that the $a_i$ are cocycles, $[a_i]= u_i$, and $da_{i,j} = a_i a_j$.
It follows that $a_1 a_{2,3} + a_{1,2}a_3$ is a $2$-cocycle and
$\langle u_1,u_2, u_3 \rangle$ then denotes the subset of
$H^2(A)$ consisting of the cohomology classes 
$[a_1 a_{2,3} + a_{1,2}a_3]$ for some choice of
the cochains $a_i, a_{i,j}$. 
Due to the ambiguity in the choice of the cocycles $a_i$,
and cochains $a_{i,j}$  
the Massey triple product $\langle u_1, u_2, u_3\rangle$ may 
be viewed as a  coset in 
$H^2(A) \slash(u_1 \cup H^1(A)+H^1(A)\cup u_3)$. 
A Massey product is said to {\em vanish}\/ if it is defined 
and  contains the element $0\in H^2(A)$.

If $\varphi\colon A\to B$ is a dga morphism, and 
$\varphi^*\colon H^*(A)\to H^*(B)$ is the induced morphism 
in cohomology, then 
\begin{equation}
\label{eq:massey-func}
\varphi^*(\langle u_1, u_2, u_3\rangle) 
\subseteq \langle \varphi^*(u_1), \varphi^*(u_2), \varphi^*(u_3)\rangle\, .
\end{equation}
In general, this inclusion is strict, but, if $\varphi$ is a $1$-quasi-isomorphism, 
then \eqref{eq:massey-func} holds as equality. 
Now, if $A$ is a dga with differential $d\colon A^1\to A^2$ equal 
to zero, then all Massey triple products $\langle u_1,u_2, u_3 \rangle$ 
vanish. Consequently, Massey triple products are an obstruction to 
$1$-formality. More precisely, if a dga $(A,d)$ is $1$-formal, and 
a triple product $\langle u_1,u_2, u_3 \rangle\subset H^2(A)$ can 
be defined, then it must contain the element $0\in H^2(A)$.

These notions extend to  higher-order Massey products (in degree one).  
Suppose $u_1,\dots , u_n$ ($n\ge 3$) are elements in $H^1(A)$.   
The $n$-fold product $\langle u_1,\dots, u_n\rangle$ is then 
defined to be the set of elements in $H^2(A)$ 
represented by cocycles of the form
$a_{1}a_{2,n}+a_{1,2}a_{3,n}+\cdots+a_{1,n-1}a_{n}$, 
where the $a_i$ are cocycle representatives of the $u_i$,
and for $i<j$ the $a_{i,j}$ are cochains in $A^1$ which satisfy 
$da_{i,j} = a_{i}a_{i+1,j}+ a_{i,i+1}a_{i+2,j}+\cdots+a_{i,j-1}a_{j}$ 
for $1 \le i < j\le n$ and $(i,j)\ne (1,n)$.

In subsequent work \cite{Porter-Suciu-2021-3}, we develop 
a theory within which generalized Massey triple and higher order products
lead to invariants of $1$-equivalence 
for dgas, and apply this to homotopy classification problems
for spaces. Example \ref{ex:mtp-r} illustrates a particular 
instance of this approach.

\section{Steenrod $\cup_i$-products}
\label{sect:steenrod-cup}

In this section, we review the properties of Steenrod $\cup_i$
products and then focus on properties of the $\cup_1$ product of
elements in dimension 1 and 2 to define cup-one differential graded
algebras. 

\subsection{The  $\cup_i$ products}
\label{subsec:cup-i}
We now enrich the notion of dga with some extra structure. 
We start with a definition due to Steenrod \cite{Steenrod}, 
as developed in subsequent work of Hirsch \cite{Hirsch}, 
Kadeishvili \cite{Kadeishvili-2003, Kadeishvili-2004}, 
Saneblidze \cite{Saneblidze}, and Franz \cite{Franz-2003, Franz-2019}. 

\begin{definition}
\label{def:cup-i-algebra}
A {\em dga with Steenrod products}\/ consists of an $R$-dga $(A,\cup, d)$ 
endowed with $R$-linear maps, $\cup_i\colon A^p\otimes_{R}  A^q\to A^{p+q-i}$, 
which coincide with the usual cup product when $i=0$, 
vanish if $p<i$ or $q<i$, and satisfy the identities 
\begin{align}
\label{eq:cup-i-steenrod}
&d(a \cup_i b)  = (-1)^{\abs{a}+\abs{b}-i} a \cup_{i-1} b +
(-1)^{\abs{a}\abs{b}+\abs{a}+\abs{b}} b\cup_{i-1} a 
+ da \cup_i b + (-1)^{\abs{a}} a \cup_i db
\\
\label{eq:cup-i-hirsch}
&(a \cup b) \cup_1 c = a \cup (b\cup_1 c) + 
(-1)^{\abs{b}(\abs{c}-1)} (a \cup_1 c)\cup b 
\end{align}
for all homogeneous elements $a,b,c\in A$.  
\end{definition}
We shall refer to \eqref{eq:cup-i-steenrod} as the ``Steenrod identities" 
and to \eqref{eq:cup-i-hirsch}, with the cup product to the left of the
cup-one product, as the ``left Hirsch identity". 

\begin{lemma}
\label{lem:gr-comm}
If $A$ is a dga with Steenrod products, then its cohomology algebra, $H^*(A)$, 
is graded-commutative.
\end{lemma}

\begin{proof}
Let $a\in Z^p(A)$ and $b\in Z^q(A)$ be two cocycles.  Then, by \eqref{eq:cup-i-steenrod}, 
$d(a \cup_1 b)  = (-1)^{p+q} \big(a b + (-1)^{pq} b a \big)$. Hence, 
$[a][b] + (-1)^{pq} [b] [a]=0$, and the claim is proved. 
\end{proof}

For $a,b,c\in A^1$, the identities \eqref{eq:cup-i-steenrod} and 
\eqref{eq:cup-i-hirsch} become
\begin{align}
\label{eq:steenrod-1c}
&d(a \cup_1 b)  = -a\cup b- b\cup a  + da \cup_1 b - a\cup_1 db,
\\
\label{eq:hirsch-1c}
&(a \cup b) \cup_1 c = a \cup (b\cup_1 c) + (a \cup_1 c)\cup b .
\end{align} 

In particular, if $a,b\in Z^1(A)$ are $1$-cocyles, then
\begin{equation}
\label{eq:cup-1-1}
d(a \cup_1 b)  = -( a\cup b + b\cup a).
\end{equation}
Thus, the operation 
$\cup_1\colon Z^1(A)\otimes_R  Z^1(A)\to A^2$ 
provides an explicit witness for the non-commutativity 
of the multiplication map 
$\cup\colon Z^1(A)\otimes_R  Z^1(A)\to Z^2(A)$ and shows
that $uv = -vu$ for elements $u,v \in H^1(A)$.  

\begin{remark}
\label{rem:non-assoc}
In general, the $\cup_i$-operations are not associative, even for $i=1$, 
see for instance \cite{Kadeishvili-2003, Kadeishvili-2004,Lawson}. 
Nevertheless, as we shall see in  sections  \ref{sect:cochain}--\ref{sect:non-comm},  
the operation $\cup_1\colon A^1 \otimes_R A^1 \to A^1$ is both associative and 
commutative for our motivating examples.
\end{remark}

\subsection{Graded algebras with cup-one products}
\label{subsec:graded-cup1}

Let $A=\bigoplus_{p\ge 0} A^{p}$ be a graded algebra over a commutative ring $R$.  
Recall that we are assuming that the structure map, $R\to A$, which sends $1_R$ 
to $1_A$, is injective, so that $R$ as a subring of $A^0$. In what follows, we let 
$D^2(A)$ denote the \emph{decomposables}\/ in $A^2$; that is, the $R$-submodule 
of $A^2$ spanned by all elements of the form $a\cup b$, with $a,b\in A^1$.

\begin{definition}
\label{def:gr-cup1}
A  \emph{graded algebra with cup-one products}\/ is a graded $R$-algebra
$A$ with cup-one product maps, $\cup_1\colon A^p \otimes_R A^1 \to A^{p}$ 
for $p=1, 2$, such that
\begin{enumerate}[label=(\roman*), itemsep=2pt]
\item  \label{axi}
The left Hirsch identity \eqref{eq:hirsch-1c} is satisfied for all $a,b,c\in A^1$. 
\item  \label{axii}
The cup-one map $\cup_1\colon A^1 \otimes_R A^1 \to A^1$ gives 
the $R$-submodule $R \oplus A^1\subset A^0\oplus A^1$ the structure 
of a commutative ring.
\end{enumerate}
\end{definition}

In \ref{axii}, the $R$-module $R \oplus A^1$ already 
possesses three partial multiplication maps, given by   
the product in the ring $R$, viewed as a map $R \otimes_R R \to R$, 
and the cup product maps
$A^1 \otimes_R R \to A^1$ and $R \otimes_R A^1 \to A^1$, 
which make $A^1$ into an $R$-bimodule. 
The added structure provided in Definition \ref{def:gr-cup1} is 
the cup-one product map, $\cup_1\colon A^1 \otimes_R A^1 \to A^1$, 
which meshes with the aforementioned maps to give a multiplication 
map, $(R \oplus A^{1}) \otimes_R (R\oplus A^1) \to R\oplus A^1$. 
Let us highlight the fact that axiom \ref{axii} requires that the cup-one product 
$\cup_1\colon A^1 \otimes_R A^1 \to A^1$ be both associative 
and commutative.

Note that, if $A$ is a {\em connected}\/ graded algebra with cup-one 
products, then $A^{\le 1}=A^0\oplus A^1$ acquires the structure of 
a commutative ring.  As can be seen in Example \ref{ex:interval}, 
though, this is not necessarily the case if $A$ is not connected.

For a graded algebra with cup-one products, $A$, 
the left Hirsch identity expresses the cup-one product 
$D^2(A)\otimes_R A^1 \to A^2$ in terms of cup-one products 
from $A^1 \otimes_R A^1$ to $A^1$. This raises the question 
of whether there is an analogous formula that expresses 
$a\cup_1 (b\cup c)$ in terms of cup-one products from 
$A^1 \otimes_R A^1$ to $A^1$.

Hirsch, \cite{Hirsch}, in the context of cochain algebras,
and Abbassi, \cite[Remarque~4.5 (1)]{Abbassi-2013}, in
the context of non-commutative differential forms, give
examples to show that there is no general identity relating  
$a \cup_1 (b \cup c)$, $a \cup_1 b$, and $a \cup_1 c$; 
see Proposition \ref{prop:no-hirsch} and Example \ref{ex:abbassi} below.
Nevertheless, these observations do not preclude the possibility 
of a formula for $a \cup_1 (b \cup c)$ in terms of cup-one products 
from $A^1 \otimes_R A^1$ to $A^1$.  We will give in 
Equation \eqref{eq:otid}  such a formula, that we call
the {\em right Hirsch identity}, in the case when 
the differential of $a$ is decomposable.  

\begin{example}
\label{ex:trivial-cup1}
There is a trivial way to impose a cup-one structure on a dga $(A,d)$: 
simply declare all the maps $\cup_1\colon A^p \otimes_R A^q \to A^{p+q-1}$ 
to be the zero maps. 
\end{example}

\subsection{Cup-one differential graded algebras}
\label{subsec:cupd-formula}
The goal for the rest of this section is to give a formula for the differential
of $1$-cochains that involves only cup products and cup-one
products of $1$-cochains. This leads to the following definition, 
which will play an important role in the sequel.

\begin{definition}
\label{def:cup-one-d-algebra}
A differential graded $R$-algebra $(A,d)$ is called a 
\emph{$\cup_1$-differential graded algebra}\/ (for short, $\cup_1$-dga) if 
the following conditions hold. 
\begin{enumerate}[label=(\roman*), itemsep=2pt]
\item \label{cup1-1}
$A$ is a graded $R$-algebra with cup-one products.
\item \label{cup1-2}
There is an $R$-linear map $\circ \colon D^2(A) \otimes_R D^2(A) \to D^2(A)$ 
such that 
\begin{equation}
\label{eq:circ-op}
(u\cup v) \circ (w\cup z) = (u \cup_1 w) \cup (v \cup_1 z)
\end{equation}
for all  $u,v,w,z\in A^1$. 
\item \label{cup1-3}
The differential $d$ satisfies the ``$\cupd$ formula,"
\begin{equation}
\label{eq:c1d}
d(a \cup_1 b) = - a \cup b - b \cup a
+ da \cup_1 b + db \cup_1 a - da \circ db,
\end{equation}
for all $a,b\in A^1$ with $da,db \in D^2(A)$.
\end{enumerate}
\end{definition}

\begin{example}
\label{ex:trivial-cup1-bis}
Let $(A,d)$ be a graded-commutative dga. Put on $A$ the 
trivial cup-one structure from Example \ref{ex:trivial-cup1} 
and set the $\circ$ product equal to the zero map. 
It is readily verified that both the $\cupd$ formula and the 
Steenrod identities are satisfied in this case, and so $(A,d)$ is 
a $\cup_1$-dga.
\end{example}

The next lemma gives a condition for the $\cupd$ formula to 
hold in a dga with Steenrod products.

\begin{lemma}
\label{lem:sth-cup1}
Let $(A,d)$ be dga such that $A$ is a graded algebra with 
cup-one products which is endowed with a map 
$\circ \colon D^2(A) \otimes_R D^2(A) \to D^2(A)$ 
such that formula \eqref{eq:circ-op} holds. 
Suppose the following two conditions are also satisfied.
\begin{enumerate}
\item 
\label{eq:hirsch-deg1}
The right Hirsch identity is satisfied in degree $1$; 
that is, for all $a,b,c\in A^1$ with $da \in D^2(A)$, 
\begin{equation}
\label{eq:otid}
a \cup_1 (b\cup c)
			= da \circ (b \cup c) - (b\cup c) \cup_1 a.
\end{equation}
\item The Steenrod identity is satisfied in degree $1$; that is, 
for all $a,b\in A^1$, 
\begin{equation}
\label{eq:steenrod-again}
d(a \cup_1 b) = - a \cup b - b \cup a + da \cup_1 b - a \cup_1 db.
\end{equation}
\end{enumerate}
Then the $\cupd$ formula \eqref{eq:c1d} holds, and so $(A,d)$ 
is a $\cup_1$-dga.
In particular, if $A$ is a dga with Steenrod products and 
equations and \eqref{eq:circ-op} and \eqref{eq:otid} are satisfied, 
then $A$ is a $\cup_1$-dga.
\end{lemma}

\begin{proof}
Let $a,b\in A^1$ with $da,db\in D^2(A)$.
From the assumptions that $da$ and $db$ are decomposable it
follows from equation \eqref{eq:otid} and the linearity
of the $\cup_1$ and $\circ$ products that
\begin{equation}
\label{eq:a-cup1-db}
a \cup_1 db = (da)\circ (db) - db \cup_1 a. 
\end{equation}
Combining this equation with \eqref{eq:steenrod-again}, 
the claim follows.
\end{proof}

\begin{question}
\label{quest:coho-cup1}
Suppose $(A,d)$ is a $\cup_1$-dga. Does the cohomology algebra 
$H^*(A)$, with cup-product inherited from $A$ and with zero differential, also
inherit in a natural way a $\cup_1$-product from $A$ that makes it into a 
graded algebra with cup-one products?
\end{question}

\subsection{Functoriality}
\label{subsec:cup1-functor}
A morphism of $\cup_1$-dgas is a dga map $\varphi\colon A \to B$ between 
two such objects which commutes with the cup-one products, that is, 
$\varphi(a_1\cup_1 a_2) = \varphi(a_1) \cup_1 \varphi(a_2)$, 
for all $a_1, a_2\in A^1$.  We denote the set of all such morphisms 
by $\Hom_1( A, B )$. Clearly, we have an inclusion 
$\Hom_1( A, B )\subseteq \Hom( A, B )$, but in general this 
inclusion is strict, as illustrated in the following example.

\begin{example}
\label{ex:hom-hom1}
Let $A=\bigwedge(a)$ be the exterior algebra over $R=\Z$
on a single generator $a$ in degree $1$, with $da=0$ and with
$ a\cup_1 a = a$. Then $\Hom(A,A)\cong \Z$ while $\Hom_1(A,A)$ 
contains only two elements, namely, the zero map and the identity.
\end{example}

The notions of $1$-quasi-isomorphism and $1$-equivalence 
are defined in the category of $\cup_1$-dgas  
exactly as in the category of dgas.  Note though that two 
$\cup_1$-dgas can be $1$-equivalent as dgas, 
but not as $\cup_1$-dgas.  For instance, take $A$ to 
be the $\cup_1$-dga from Example \ref{ex:hom-hom1}, 
and $A'$ to be the same dga, but with trivial cup-one structure, 
$a\cup_1 a=0$. 

It would be interesting to have a definition of $1$-formality in 
this context, but this depends on answering Question \ref{quest:coho-cup1}
first. If this could be done, it would also be interesting to decide 
whether it is possible that a $\cup_1$-dga may be $1$-formal 
as a dga, but not as a $\cup_1$-dga.

\section{Cochain algebras}
\label{sect:cochain}
 The first of our two motivating examples for the definition
 of cup-one dgas is the cochain complex of a space.
 In this section we review the notion of a $\Delta$-complex
 and the cochain complex of a $\Delta$-set. 
 Theorem \ref{thm:quasi-iso} relates the $n$-type of 
 CW complexes to equivalences of cochain complexes.
 
\subsection{$\Delta$-sets and $\Delta$-complexes}
\label{subsec:delta}
We start the section by reviewing the notion of a $\Delta$-complex, in 
the sense of Rourke and Sanderson \cite{Rourke-Sanderson}; 
see also Hatcher \cite{Hatcher} and Friedman \cite{Friedman}. 
We will view such a complex as the geometric realization of the 
corresponding $\Delta$-set, cf.~\cite{Friedman}.

An (abstract) $n$-simplex $\Delta^n$ is simply a finite ordered set, $(0,1, \dots, n)$. 
The face maps $d_i\colon \Delta^n\to \Delta^{n-1}$ 
for $n \ge 1$ are given by omitting the $i$-th 
element in the set; that is, 
\begin{equation}
\label{eq:diff-simplicial}
d_i(0,\dots,n)=(0,\dots, \hat{i}, \dots ,n) .
\end{equation}
These maps satisfy $d_i d_j=d_{j-1}d_i$ whenever 
$0\le i<j\le n$ and $n \ge 2$.
The geometric realization of the simplex, $\abs{\Delta^n}$, is the convex hull of $n+1$ 
affinely independent vectors in $\R^{n+1}$, endowed with the subspace topology; the 
face maps induce continuous maps, $d_i\colon \abs{\Delta^n}\to \abs{\Delta}^{n-1}$.

More generally, a {\em $\Delta$-set}\/ consists of a sequence of sets $\{X_n\}_{n\ge 0}$ 
and maps $d_i\colon X_{n+1}\to X_n$  for each $0\le i\le n+1$ such that 
$d_i d_j=d_{j-1}d_i$ whenever $i<j$.  This is the generalization of the 
notion of ordered (abstract) simplicial complex, where the sets $X_n$ 
are the sets of $n$-simplices and the maps $d_i$ are the face maps. Note that the singular simplices in a topological space
form a $\Delta$-set.

The geometric realization of a $\Delta$-set is the topological space
\begin{equation}
\label{eq:delta-realize}
\abs{X} = \coprod_{n\ge 0} X_n \times \abs{\Delta^n}/\!\sim, 
\end{equation}
where $\sim$ is the equivalence relation generated by $(x,\partial_i(p))\sim (d_i(x), p)$ 
for $x\in X_{n+1}$,  $p\in \abs{\Delta^n}$, and $0\le i\le n$.  Such a space is called 
a {\em $\Delta$-complex}, and can be viewed either as a special kind of 
CW-complex, or a generalized simplicial complex.

The assignment $X\leadsto \abs{X}$ is functorial: 
if $f\colon X\to Y$ is a map of $\Delta$-sets (i.e., $f$ is a family of maps 
$f_n\colon X_n \to Y_n$ commuting with the face maps), there is an 
obvious realization, $\abs{f}\colon \abs{X} \to \abs{Y}$, and this is a  
(continuous) map of $\Delta$-complexes. We will often abuse notation 
and mistake an (abstract) $\Delta$-complex for its geometric realization, 
$\abs{X}$, and likewise for maps between these objects.

\subsection{Cochain algebras of $\Delta$-complexes}
\label{subsec:delta-cochains}

The chain complex of a $\Delta$-set $X$, denoted $C_*(X)$, coincides with the 
simplicial chain complex of its geometric realization: for each $n\ge 0$, the chain 
group $C_n(X)$ is the free abelian group on $X_n$, while the boundary 
maps  $\partial_n\colon C_n(X)\to C_{n-1}(X)$ are the $\Z$-linear maps given 
on basis elements by $\partial_n=\sum_{i=0}^n (-1)^i d_i$, where $d_i$ is given 
by formula \eqref{eq:diff-simplicial}. The cochain complex 
$C^*(X)=\big(C^n(X),\delta^n\big)_{n\ge 0}$ is defined similarly 
with $(\delta u)(x) = u(\partial x)$ for
$u \in C^n(X)$ and $x \in C_{n+1}(X)$.

More generally, we let $C=C^*(X;R)$ be the cochain complex of $X$ 
with coefficients in a commutative ring $R$ with unit $1$. This $R$-module 
acquires the structure of an $R$-algebra, with multiplication given by the 
cup-product of cochains. More precisely, if $u\in C^p(X;R)$ and $v \in C^q(X;R)$,
then $u \cup v  \in C^{p+q}(X;R)$ is the cochain given by 
\begin{equation}
\label{eq:cup-cochains}
u \cup v \, ([ 0,1,\ldots, p+q]) = 	u([ 0, \ldots, p])\cdot v([p, \ldots, p+q])
\end{equation}	
where $[ a_0, a_1, \ldots, a_k]$ denotes a $k$-simplex on the indicated vertices 
and $\,\cdot\,$ denotes the product in $R$.  Moreover, the unit $1_C\in C^0(X;R)$ 
is the $0$-cochain which takes the value $1_R$ on each $0$-simplex; clearly, 
the structure map $R\to C$ which takes $1_R$ to $1_C$ is injective. 

Two maps of spaces, $f, g\colon X\to Y$, are said to be {\em 
$n$-homotopic}\/ (in the sense of \cite{Wh}), if $f\circ h\simeq g\circ h$, 
for every map $h\colon K\to X$ from a CW-complex $K$ of dimension 
at most $n$. A map $f\colon X\to Y$ is an $n$-homotopy equivalence 
(for some $n\ge 1$) if it admits an $n$-homotopy inverse. If such a map 
$f$ exists, one says that $X$ and $Y$ have the same {\em $n$-homotopy type}, 
written $X\simeq_n Y$.  Two CW-complexes, $X$ and $Y$, are said to be 
of the same {\em $n$-type}\/  if $X^{(n)}\simeq_{n-1} Y^{(n)}$. In particular, 
any two connected CW-complexes have the same $1$-type, and 
they have the same $2$-type if and only if their fundamental groups are 
isomorphic.

Continuous maps between CW-complexes (or $\Delta$-complexes, 
or simplicial complexes) can be approximated up to homotopy by maps 
that respect cellular structures; we will do that when needed, without 
further mention.  A map $f\colon X\to Y$ between two $\Delta$-complexes 
induces a dga map, $f^{\sharp} \colon C^\ast(Y;R) \to C^\ast(X;R)$, 
between the respective cochain algebras, and thus a morphism,  
$f^{\ast}\colon H^*(Y;R)\to H^*(X;R)$, between their 
cohomology algebras. If $f$ is a homotopy equivalence, 
then $f^{\sharp}$ is a quasi-isomorphism of $R$-dgas. 
More generally, we have the following theorem, which provides 
the bridge from topology to algebra in this context. Although 
surely known to specialists, we could not find a proof in the 
literature, so we include one here.

\begin{theorem}
\label{thm:quasi-iso}
If $X$ and $Y$ are CW-complexes of the same $n$-type, then the 
cochain algebras $C^\ast(X;R)$ and $C^\ast(Y;R)$ are $(n-1)$-equivalent.
\end{theorem}

\begin{proof}
By assumption, $X^{(n)}\simeq_{n-1} Y^{(n)}$. Hence, 
by \cite[Theorem 6]{Wh}, there is a homotopy equivalence, $f$, from 
$\overline{X}^{(n)}  = X^{(n)} \vee \bigvee_{i\in I} S^n_i$ 
to $\overline{Y}^{(n)} =Y^{(n)}\vee \bigvee_{j\in J} S^n_j$, for some 
indexing sets $I$ and $J$. Let $q_X\colon \overline{X}^{(n)} \to X^{(n)}$ and 
$q_Y\colon \overline{Y}^{(n)} \to Y^{(n)}$ be the maps that collapse the 
wedges of $n$-spheres to the basepoint of the wedge. 
\begin{equation}
\label{eq:wh}
\begin{tikzcd}
\overline{X}^{(n)} \ar[d, "q_X" ] \arrow{r}{f}[swap]{\simeq}
& \overline{Y}^{(n)} \ar[d, "q_Y"]\\
X^{(n)} &  Y^{(n)}
\end{tikzcd}
\quad\text{\larger[1.5]$\leadsto$} \quad
\begin{tikzcd}
C^\ast(\overline{X}^{(n)};R)
& C^\ast(\overline{Y}^{(n)};R)\phantom{.} 
\arrow{l}{\simeq}[swap, pos=0.4]{f^{\sharp}} \\
C^\ast(X^{(n)},R) \ar[u, "q_X^{\sharp}"' ] 
&  C^\ast(Y^{(n)},R).\ar[u, "q_Y^{\sharp}"']
\end{tikzcd}
\end{equation}

Clearly, the map 
$f^{\sharp} \colon C^\ast(\overline{Y}^{(n)};R) \to C^\ast(\overline{X}^{(n)};R)$ 
is a quasi-isomorphism, whereas the maps 
$q_X^{\sharp} \colon  C^\ast(X^{(n)};R) \to C^\ast(\overline{X}^{(n)};R)$ 
and $q_Y^{\sharp} \colon C^\ast(Y^{(n)};R) \to C^\ast(\overline{Y}^{(n)};R)$ 
are $(n-1)$-quasi-isomorphisms. Therefore, 
$C^\ast(X^{(n)};R)$ and $C^\ast(Y^{(n)};R)$ are $(n-1)$-equivalent. 
Hence, $C^\ast(X;R)$ and $C^\ast(Y;R)$ are also $(n-1)$-equivalent.
\end{proof}

In particular, if $X$ and $Y$ are connected CW-complexes of the same 
$2$-type, i.e., if $\pi_1(X)\cong \pi_1(Y)$, then $C^\ast(X;R)$ is 
$1$-equivalent to $C^\ast(Y;R)$.

\subsection{The Steenrod $\cup_i$ operations on cochains}
\label{subsec:steenrod-cup-i}
Let $X$ be a $\Delta$-complex, and let $A=(C^*(X;R),\cup,\delta)$ be its 
cochain algebra  with coefficients in a commutative ring $R$. 
In a seminal paper from 1947, Steenrod \cite{Steenrod} introduced 
operations $\cup_i\colon A^p\otimes_{R}  A^q\to A^{p+q-i}$ 
that now bear his name. For $i=0$, the $\cup_0$ operation coincides 
with the usual cup product, while  if $p<i$ or $q<i$, then $\cup_i=0$. 
Crucially, Steenrod's $\cup_i$ products satisfy the identities \eqref{eq:cup-i-steenrod}. 
Several years later, Hirsch \cite{Hirsch} showed that the identities \eqref{eq:cup-i-hirsch} 
also hold. In our terminology from Definition \ref{def:cup-i-algebra}, the cochain algebra 
$A$ is a dga with Steenrod products. 

The construction of the $\cup_i$ products is functorial, in the following 
sense. Let $f\colon X\to Y$ be a map of $\Delta$-complexes. Without 
loss of generality, we may take baycentric subdivisons on $X$ and $Y$. 
Now put partial orders on the resulting sets of vertices by assigning to each
vertex the dimension of the simplex for which it is the barycenter of; note that 
this defines a linear order on each simplex of $X$ and $Y$. With these constructions, 
the map $f$ induces a chain map $f^{\sharp} \colon C^\ast(Y) \to C^\ast(X)$ 
that preserves the ordering of the vertices. Thus, by 
\cite[Theorem 3.1]{Steenrod}, the map $f^{\sharp}$ also  
commutes with the $\cup$ and $\cup_i$ products.  

The following is the definition from \cite{Steenrod}
of the $\cup_1$-products of cochains $u\in A^p$, $v \in A^q$, 
when evaluated on a simplex,
\begin{align}
\label{eq:cup1}
u \cup_1\! v \, ([ 0,1,\ldots, p+q-1])&=\sum_{j=0}^{p-1} (-1)^{(p-j)(q+1)}
		u([ 0, \ldots, j, j+q, \ldots, p+q-1])\cdot
\\[-6pt] \notag
& \hspace{2in}v([j, \ldots, j+q])\, .
\end{align}	
Clearly, $u\cup_1 v=0$ if either $u$ or $v$ is a $0$-cochain. 
When both $u$ and $v$ are $1$-cochains, formulas \eqref{eq:cup-cochains} 
and \eqref{eq:cup1} simplify to
\begin{align}
\label{eq:cup-simplicial}
(u \cup v)(s) &= u(e_1)\cdot v(e_2) , \\
\label{eq:cup1-simplicial}
(u \cup_1\! v)(e) &=  u(e)\cdot v(e) , 
\end{align} 
where in \eqref{eq:cup-simplicial}, $s=[i,j,k]$ is a $2$-simplex with front face $e_1=[i,j]$ 
and back face $e_2=[j,k]$, while in \eqref{eq:cup1-simplicial}, $e$ is a $1$-simplex. 
In particular, the $\cup_1$-product on $C^1(X;R)$ is both associative 
and commutative. Therefore, the $R$-module 
$R\oplus C^{1}(X;R)$ naturally acquires 
the structure of a commutative ring with unit. Since, 
as mentioned previously, the left Hirsch identity \eqref{eq:hirsch-1c} 
holds, the cochain algebra $C^*(X;R)$ is a graded algebra with cup-one products, 
in the sense of Definition \ref{def:gr-cup1}.

\begin{example}
\label{ex:interval}
Let $I$ be the closed interval $[0, 1]$, viewed as a simplicial complex 
in the usual way, and let $C=C^\ast(I;R)$ be its cochain algebra over $R$.  
Then $C^0 = R \oplus R$ with generators $t_0, t_1$ corresponding
to the endpoints $0$ and $1$, and $C^1 = R$ with generator $u$.
The differential $d\colon C^0 \to C^1$ is given by
$d t_0 = -u$ and $d t_1 = u$, while the multiplication is given 
on generators by $t_i  t_j = \delta_{ij} t_i$, 
$t_0  u = u  t_1 =u$, and $u t_0 = t_1  u = 0$. Note that 
the cocycle $t_0+t_1$ is the unit of $C$, and that multiplication 
of $0$- and $1$-cocycles is not commutative, e.g., $t_0u\ne ut_0$.   
Furthermore, $H^*(C)=R$, concentrated in degree $0$.  
Finally, the cup-one product $C^1 \otimes_R C^1 \to C^1$ is 
given by $u \cup_1 u = u$.
\end{example}

\begin{example}
\label{ex:torus-cup}
Let $X=S^1\times S^1$ be the $2$-torus, and  
consider the $\Delta$-complex structure on $X$ 
depicted in Figure \ref{fig:torus}. The attaching 
maps of the standard $2$-simplex 
are indicated by the numbers on the vertices and the requirement
that the correspondence preserves the ordering. Edges are oriented
from the lower-numbered vertex to higher-numbered vertex.
The $1$-cochains are represented by 
$1$-cells---with a transverse orientation---in a graph whose
edges are transverse to the edges of the $\Delta$-complex with
vertices of the graph contained in the $2$-cells of the $\Delta$-complex.
The value of a $1$-cochain, $u$, on an oriented $1$-cell, $e$, is 
$0$ if $u$ and $e$ do not intersect.
If $u$ and $e$ intersect, then $u(e) =1$ if at the point
of intersection the transverse orientation of $u$ agrees with the
orientation of $e$; otherwise, $u(e) = -1$.
Now write $\pi_1(X)=\Z^2=\langle x_1, x_2\mid 
x^{}_1x^{}_2x_1^{-1}x_2^{-1}\rangle$, 
and identify $H^*(X)$ with $\bigwedge^*(v_1,v_2)$, 
the exterior algebra on two generators  in degree $1$. 
Let $a_1$ and $a_2$ be the $1$-cocycles shown in the figure, so that  
$[a_i]=v_i$, and let $b$ be the indicated $1$-cochain. 
Then $b = a_1 \cup_1 a_2$ and $d(b) = -a_1 \cup a_2 - a_2 \cup a_1$.

\begin{figure}
\begin{tikzpicture}[scale=2]
\node [left] at (0,1) {${x}_{1}$};
\node [below] at (1,0) {${x}_{2}$};
\node [right] at (2,1) {${x}_{1}$};
\node [above] at (1,2) {${x}_{2}$};
\draw[->] (0,1.5)--(0,1);
\draw[->] (.75, 0) -- (1,0);
\draw[->] (2,1.5) -- (2,1);
\draw[->] (.75,2) -- (1,2);

\draw (0,0) -- (0,2) -- (2,2) -- (2,0) -- (0,0);
\draw (0,0) -- (2,2);
\draw [red, ultra thick] (0,1.25) -- (2,1.25);
\draw [blue, ultra thick]  (0.75, 0) -- (0.75, 2); 
\draw [red, thick, <-] (0.3,1.1) -- (0.3,1.4]); 
\draw [blue, thick, ->] (0.6,0.3) -- (0.9,0.3); 
\draw [brown, ultra thick] (0.85,1.15) -- (1.15,0.85);
\draw [brown, thick, <-] (0.95,0.85) -- (1.15, 1.05]); 
\node [red, thick, below] at (0.33,1.1) {$a_1$}; 
\node [blue, thick, below] at (0.5,0.35) {$a_2$}; 
\node [brown, thick, below] at (0.92,0.9) {$b$}; 
\node [above left] at (0,2) {\footnotesize $1$};
\node [below left] at (0,0) {\footnotesize $2$};
\node [below right] at (2,0) {\footnotesize $4$};
\node [above right] at (2,2) {\footnotesize $3$};
\end{tikzpicture}
\caption{Cochains on a torus}
\label{fig:torus}
\end{figure}
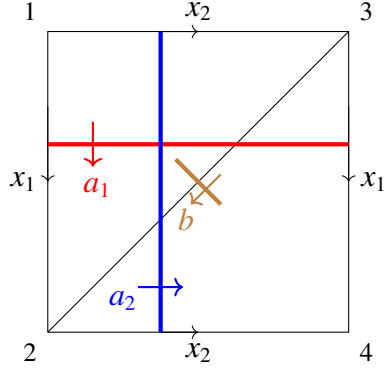
\end{example}

\subsection{Cochain algebras as cup-one differential graded algebras}
\label{subsec:cochains-cup1d}

We are now ready to state and prove the main result of this section. 
Let $X$ be a non-empty $\Delta$-complex, and let $R$ be a unital 
commutative ring.

\begin{theorem}
\label{thm:cochain-cup1d}
Let $A = C^\ast(X;R)$ be a cochain algebra. Then $A$ is a $\cup_1$-dga.
\end{theorem}

\begin{proof}
As mentioned previously, $C^*(X;R)$ is a graded algebra 
with cup-one products. Furthermore, it is a differential graded 
algebra, and the Steenrod identities \eqref{eq:cup-i-steenrod} 
hold in full generality; in particular, \eqref{eq:steenrod-again} 
holds.  In view of Lemma \ref{lem:sth-cup1}, we only need to 
show that there is a well-defined $\circ$ map satisfying
\eqref{eq:circ-op} and that the right Hirsch identity, 
equation \eqref{eq:otid}, also holds.

From formula \eqref{eq:cup1-simplicial} it follows that the $\cup_1$-product 
is both associative and commutative for $1$-cochains. Therefore, 
$R \oplus C^1(X;R)$ naturally acquires the structure of a commutative ring
with unit, in the fashion outlined right after Definition \ref{def:gr-cup1}.
From formulas \eqref{eq:cup-simplicial}--\eqref{eq:cup1-simplicial}, 
it follows that for $1$-cochains $u, v, w, u_1,$ and  $u_2$ and for 
a $2$-simplex $s$ in $X$ we have that 
\begin{align}
 \label{eq:derive-1}
u(e_1)(v \cup w)(s) & = ((u\cup_1 v) \cup w)(s),\\
\label{eq:derive-2}
u(e_2)(v \cup w)(s)& = (v \cup (u \cup_1 w))(s),\\
\label{eq:derive-3}
(u_1 \cup u_2)(s) \cdot (v \cup w)(s)
	& = ((u_1 \cup_1 v) \cup (u_2 \cup_1 w))(s).
\end{align}

Define now an $R$-linear map 
$\circ\colon A^2 \otimes_R A^2 \to A^2$ 
by setting
\begin{equation}
\label{eq:circ-cochains}
(v \circ w)(s) = v(s)\cdot w(s)
\end{equation}
for any $2$-cochains $v,w$ and any $2$-simplex $s$. 
It follows from \eqref{eq:cup-simplicial},  \eqref{eq:cup1-simplicial}, 
and \eqref{eq:derive-3} that the map above restricts to a map 
$\circ\colon D^2(A) \otimes_R D^2(A) \to D^2(A)$ which 
obeys formula \eqref{eq:circ-op}.

The next step is to show that the right Hirsch identity holds; 
that is, if $du = \sum_{i} u_{1,i} \cup u_{2,i}$, then
\begin{equation}
\label{eq:S0}
u \cup_1 (v \cup w)
	= - (u\cup_1 v) \cup w 
		+ \sum_{i} (u_{1,i}\cup_1 v) \cup (u_{2,i}\cup_1 w)
		- v \cup(u\cup_1 w).
\end{equation}

By equation \eqref{eq:cup1}, the cup-one product map 
$\cup_1\colon A^1 \otimes_R A^2 \to A^2$ is given by
\begin{equation}
\label{eq:S1}
(u \cup_1 z)(s) = - u(e_3)z(s), 
\end{equation}
for $u\in A^1$ and $z \in A^2$. Since $d u  = u(e_1) + u(e_2) - u(e_3)$, 
it follows that
\begin{equation}
\label{eq:1-cup1-2}
( u \cup_1 z)(s) = (-u(e_1) - u(e_2) + du(s))\cdot z(s).
\end{equation}
For $z = v \cup w$ and $du = \sum_i u_{1,i} \cup u_{2,i}$,
it then follows from equations \eqref{eq:derive-1}, \eqref{eq:derive-2}, 
and \eqref{eq:derive-3} that
\begin{equation}
\label{eq:S4}
\bigl( u \cup_1 (v \cup w) \bigr)(s)
	= - \biggl( (u \cup_1 v) \cup w 
		- \sum_{i} (u_{1,i}\cup_1 v) \cup (u_{2,i}\cup_1 w)
		+ v \cup (u\cup_1 w) \biggr)(s).
\end{equation}
Since equation \eqref{eq:S4} holds for all $2$-simplices, $s$, 
the proof is complete.
\end{proof}

The next result is motivated by---and generalizes---an example 
due to Hirsch \cite{Hirsch}. 

\begin{proposition}
\label{prop:no-hirsch}
There is no linear combination of cup products and cup products of 
cup-one products that equals $u_3 \cup_1 (u_1 \cup u_2)$ for all $1$-cochains
$u_1, u_2, u_3\in C^1(X;R)$ and all $\Delta$-complexes $X$.
\end{proposition}

\begin{proof}
We need to show that there are no constants
$\alpha(i_1,i_2 | i_3)$, $\alpha(j_1 | j_2, j_3)$,
$\alpha (k_1,k_2 | k_3, k_4)$, and $\alpha(\ell_i|\ell_2)$
with
$1 \le i_r \le 3$, $1 \le j_r \le 3$, $1 \le k_r \le 3$, and
$1 \le \ell_r\le 3$,
such that the equation
\begin{equation}
\label{eq:no-hirsch}
\begin{split}
u_3 \cup_1 (u_1 \cup u_2)
		&= \sum_{1 \le i_\ell \le 3} 
		     \alpha(i_1, i_2 | i_3) (u_{i_1} \cup_1 u_{i_2})\cup u_{i_3}
	   + \sum_{1 \le j_\ell \le 3}
		     \alpha (j_1 | j_2,j_3) u_{j_1} \cup (u_{j_2} \cup_1 u_{j_3}) \\
		&\quad  + \sum_{1 \le k_\ell \le 3} \alpha (k_1,k_2 | k_3, k_4)
					(u_{k_1} \cup_1  u_{k_2}) \cup (u_{k_3} \cup_1 u_{k_4})
				 + \sum_{1 \le l_i \le 3} 
					\alpha(l_i | l_j) u_{l_i}\cup u_{l_j}
\end{split}
\end{equation}
holds for all $1$-cochains, $u_1, u_2, u_3$, on an arbitrary $\Delta$-complexes $X$. 

Let $X$ be the standard $2$-simplex $[0,1,2]$, with front face $e_1 = [0,1]$,
back face $e_2 = [1,2]$, and third face $e_3 = [0,2]$. Let $n$
be an integer and define $1$-cochains on $[0,1,2]$ by
\begin{align*}
u_1(e_i) &= 1\text{ if $i=1$ and $0$ otherwise},\\
u_2(e_i) &= 1\text{ if $i=2$ and $0$ otherwise},\\
u_3(e_i) &= n\text{ if $i=3$ and $0$ otherwise}.
\end{align*}
By direct computation of cup and cup-one products of these cochains,
it follows that the right hand side of equation \eqref{eq:no-hirsch}
is equal to
\begin{align*}
&\alpha(1,1 |2) (u_1 \cup_1 u_1) \cup u_2
+ \alpha(1|2,2)u_1 \cup (u_2 \cup_1 u_2)\\
& \qquad+ \alpha(1,1 | 2,2)  (u_1 \cup_1 u_1) \cup (u_2 \cup_1 u_2)
+ \alpha(1|2)u_1 \cup u_2\\
& \quad =
				\alpha(1,1 |2) 
+ \alpha(1|2,2)
+ \alpha(1,1 | 2,2) 
+ \alpha(1|2).
\end{align*}
From the definition of the cup-one product, we have that
\[
u_3 \cup_1 (u_1 \cup u_2)([0,1,2]) 
			= -u_3(e_3)\cdot (u_1\cup u_2)([0,1,2]) = -n.
\]
Thus, assuming that equation \eqref{eq:no-hirsch} holds in our 
situation, we have that
\begin{equation}
\label{eq:n-alpha}
-n  = \alpha(1,1 |2) + \alpha(1|2,2)+ \alpha(1,1 | 2,2) + \alpha(1|2)
\end{equation}
for all integers $n$. This is a contradiction, since the right hand side
of equation \eqref{eq:n-alpha} is a constant, and so we are done.
\end{proof}

\section{Non-commutative differential forms}
\label{sect:non-comm}

The second of our two motivating examples for the definition
 of cup-one dgas is the algebra of non-commutative differential
 forms. In this section we review the definition and properties 
 of non-commutative differential forms and show that they
 are cup-one differential graded algebras.

\subsection{The algebra of non-commutative differential forms}
\label{subsec:Omega}
We start this section with a brief review of the theory of non-commutative 
differential forms, as developed by H.~Cartan \cite{Cartan-1976}, 
A.~Connes \cite{Connes-1985}, and M.~Karoubi 
\cite{Karoubi-tams1995, Karoubi-asens1995}, 
as well as some more recent developments due to 
N.~Battikh \cite{Battikh-2007, Battikh-2011} and 
A.~Abbassi \cite{Abbassi-2013}. 

Let $R$ be a commutative ring with unit $1$,  and let $A$ be an $R$-algebra
with injective structure map $R\to A$. 
For each $n\ge 0$, let $T^n(A)=A\otimes_R \cdots \otimes_R A$ be the 
$(n+1)$-fold tensor product of $A$ (over $R$).  

\begin{proposition}
\label{prop:tensor-dga}
The $R$-module $T^*(A)=\bigoplus_{n\ge 0} T^n(A)$ has the structure, 
of an $R$-dga with multiplication $T^n\otimes_R T^m \to T^{n+m}$ 
and differential $D\colon T^n\to T^{n+1}$ given by 
\begin{align}
\label{eq:ta-dga-1}
&(a_0\otimes a_1 \otimes \cdots  \otimes  a_n)\cup
(b_0\otimes b_1 \otimes \cdots  \otimes  b_m) =
a_0\otimes a_1 \otimes \cdots  \otimes  a_nb_0
\otimes b_1 \otimes \cdots  \otimes  b_m,
\\
&D(a_0\otimes a_1 \otimes \cdots  \otimes  a_n)=
1\otimes a_1 \otimes \cdots  \otimes  a_n - a_0\otimes 1 \otimes \cdots  \otimes  a_n
+ \cdots \\ 
&\hspace*{3in} + (-1)^{n+1} a_0 \otimes \cdots  \otimes a_{n}\otimes  1.\notag
\end{align}
\end{proposition}

In particular, $D(a)=1\otimes a-a\otimes 1$ and 
$D(a_0\otimes a_1)=1\otimes a_0\otimes a_1-a_0\otimes 1\otimes a_1+
a_0\otimes a_1 \otimes 1$.

Now set $\Omega^0(A)=A$ and let $\Omega^1(A)$ be the kernel of the 
multiplication map $A\otimes_R A\to A$. 
Clearly, the $R$-module $\Omega^1(A)$ 
is also an $A$-bimodule. For each $n\ge 1$, we let  
\begin{equation}
\label{eq:nc-Omega}
\Omega^n(A)= \Omega^1(A) \otimes_A \cdots \otimes_A \Omega^1(A) 
\end{equation}
be the $n$-fold tensor product of $\Omega^1(A)$ over $A$. 
Since $A\otimes_A  A=A$, we have that $\Omega^n(A)\subset T^n(A)$. 
Thus, the $R$-module 
$\Omega^*(A)=\bigoplus_{n\ge 0} \Omega^n(A)$ inherits a natural $R$-dga 
structure from $T^*(A)$; we denote its differential by $d$. 

Writing $\overline{A}=A/R$, it is readily seen that the map 
$A\otimes_R \overline{A}\to \Omega^1(A)$, $x\otimes \bar{y}\mapsto xdy$ is 
an isomorphism.  Hence, 
\begin{equation}
\label{eq:omega-n-A}
\Omega^n(A)\cong A \otimes_R \overline{A} 
\otimes_R \cdots \otimes_R \overline{A} , 
\end{equation}
where the number of factors of $\overline{A}$ is $n$.
The elements of this $R$-module, known 
as {\em non-commutative forms of degree $n$}, 
can be viewed as linear combinations of elements of the form 
$a_0 da_1\cdots da_n$.  
The differential $d$ given by 
\begin{equation}
\label{eq:d-omega}
d(a_0 da_1\cdots da_n)= da_0 da_1\cdots da_n.
\end{equation}

The canonical projection, $J\colon T^*(A)\to \Omega^*(A)$, 
$a_0\otimes a_1 \otimes \cdots  \otimes  a_n\mapsto a_0 da_1\cdots da_n$ is a 
morphism of graded $R$-modules which is compatible with the differentials, but 
not with the algebra structures.  Nevertheless, $J(\alpha\beta)=J(\alpha)J(\beta)$
if $\beta$ is a cocycle in $T^*(A)$.

\subsection{Cup-$i$ products in $T(A)$}
\label{subsec:TA-cup-i}
In \cite[\S 3]{Battikh-2007}, Battikh defines $R$-linear maps 
$\cup_i \colon T^{n}(A)\otimes_R  T^m(A)\to T^{n+m-i}(A)$ 
which vanish if $i>\min(n,m)$, coincide with the multiplication 
in $T^*(A)$ if $i=0$, and satisfy Steenrod's identities, 
\begin{equation}
\label{eq:cup-i-bat}
D(\alpha\cup_i \beta)= (-1)^{n+m-i} \alpha \cup_{i-1} \beta +
(-1)^{nm+n+m} \beta\cup_{i-1} \alpha
+ D\alpha \cup_i \beta + (-1)^{n} \alpha \cup_i D\beta.
\end{equation}
While these maps are rather difficult to write down explicitly, 
here is the definition of the $\cup_1$ map. 
Let $\alpha=a_0 \otimes \cdots \otimes  a_p$ and 
$\beta=b_0 \otimes \cdots \otimes b_q$; 
then (\cite[Proposition 3.1]{Battikh-2007}):
\begin{equation}
\label{eq:cup-one-bat}
\begin{split}
&(a_0 \otimes \cdots \otimes a_p)
\cup_1
(b_0 \otimes \cdots \otimes b_q) = \sum_{i=0}^{p-1}(-1)^{(p-i)(q+1)}
a_0 \otimes \cdots \otimes a_{i-1} \otimes a_ib_0\otimes  \\
&\hspace{2.5in} b_1\otimes \cdots \otimes b_{q-1}
\otimes b_qa_{i+1}\otimes a_{i+2} \otimes \cdots
\otimes a_p.
\end{split}
\end{equation}
In particular, 
\begin{equation}
\label{eq:cup1-t-alg}
(a_0 \otimes a_1)
\cup_1
(b_0 \otimes b_1) = 
a_0 b_0 \otimes b_1a_1.
\end{equation}

Now suppose $A$ is a commutative $R$-algebra. Then the (left) 
Hirsch formula holds in all degrees (\cite[Proposition 4.4]{Abbassi-2013}); 
namely, if $\alpha_i\in T^{n_i}(A)$, then
\begin{equation}
\label{eq:hirsch-abb}
(\alpha_1 \cup \alpha_2)\cup_1 \alpha_3 =
\alpha_1 \cup (\alpha_2\cup_1 \alpha_3) +
(-1)^{n_1(n_2+1)} (\alpha_1\cup_1 \alpha_3)\cup \alpha_2.
\end{equation}
Since $T(A)$ satisfies equations \eqref{eq:cup-i-bat} and \eqref{eq:hirsch-abb}, 
it is a dga with Steenrod products, in the sense of Definition \ref{def:cup-i-algebra}.
Moreover, formula \eqref{eq:cup1-t-alg} shows that the $\cup_1$ product on $T^1(A)$ 
is associative and commutative; therefore, it gives $R\oplus T^1(A)$ the structure of 
a commutative ring with unit. Hence, $T(A)$ is a graded algebra with cup-one products, 
in the sense of Definition \ref{def:gr-cup1}.

\subsection{Cup-$i$ products in $\Omega(A)$}
\label{subsec:Omega-cup-i}
The $\cup_i$ operations on $T^\ast(A)$ induce
$\cup_i$ operations on $(\Omega^*(A),d)$ that
satisfy Steenrod's identites, cf.~\cite[Proposition 3.6]{Battikh-2007}. 
The $\cup_1$ operation on $\Omega^1(A)$ takes the explicit form 
indicated in the next lemma.

\begin{lemma}
\label{lem:cup1-omega1}
Let $a_0da_1$ and $b_0db_1$ be two elements in $\Omega^1(A)$.  Then,
\begin{align*}
a_0da_1 \cup_1 b_0db_1
&=
a_0d(a_1b_0b_1)-a_0b_1d(a_1b_0)-a_0a_1b_0db_1
					-a_0a_1d(b_0b_1) \\
		&\quad 
+a_0a_1b_1db_0 + a_0a_1b_0db_1.
\end{align*}
In particular, $da \cup_1 db = d(ab) - b da - a db$, and so 
$da \cup_1 da = d(a^2) -2a da$.
\end{lemma}

\begin{proof}
Recall that $\Omega^\ast=\Omega^*(A)$ is an $A$-bimodule. Furthermore, 
the differential $d\colon \Omega^0\to \Omega^1$ satisfies the product rule, 
$d(ab) = (da)\cdot b + a\cdot (db)$. 
The first step is to write $da \cdot b$ as an element in $\Omega^1$
in normal form (i.e., as a sum of terms of the form $x_i dy_i$).  This 
can be done using the product rule, or, alternatively, the following computation:
\begin{align}
\label{eq:dab}
da \cdot b
		& = (1 \otimes a - a \otimes 1)b \notag
		 = 1 \otimes ab - a \otimes b\notag\\
		& = 1 \otimes ab - ab \otimes 1 + ab \otimes 1 - a \otimes b\\
		& = d(ab) + a(b \otimes 1 - 1 \otimes b) = d(ab) - adb. \notag
\end{align}
Therefore,
\begin{align*}
da \cup_1 db 
		& = (1 \otimes a - a \otimes 1)\cup_1 (1 \otimes b - b \otimes 1)
		 = 1 \otimes ab - b \otimes a - a \otimes b + ab \otimes 1\\
		& = 1 \otimes ab - ab \otimes 1 + ab \otimes 1-b \otimes a + 
		ba \otimes 1 - ab \otimes 1-a(1\otimes b - b \otimes 1)\\
		& = d(ab) -b(1\otimes a - a \otimes 1) - a db
		 = d(ab) - bda - a db.
\end{align*}
The general case follows along the same lines:
\begin{align*}
a_0da_1 \cup_1 b_0 db_1
		& = a_0 (da_1\cdot b_0) \cup_1 db_1\\
		& = a_0 (d(a_1b_0) - a_1db_0) \cup_1 db_1\\
		& = a_0 [d(a_1b_0)\cup_1 db_1] - a_0a_1 (db_0 \cup_1 db_1)\\
		& = a_0[d(a_1b_0b_1)-b_1d(a_1b_0)- a_1b_0db_1]-
		       a_0a_1[d(b_0b_1)-b_1db_0-b_0db_1]\\
		& = a_0d(a_1b_0b_1)-a_0b_1d(a_1b_0)-a_0a_1b_0db_1
					-a_0a_1d(b_0b_1)\\
		&\quad +a_0a_1b_1db_0 + a_0a_1b_0db_1,
\end{align*}
and we are done.
\end{proof}

\subsection{Non-commutative differential forms as cup-one dgas}
\label{subsec:noncomm-cup1d}
Let $A$ be a commutative $R$-algebra, and let $(\Omega(A),d)$ be 
the dga of non-commutative differential forms on $A$. 
Note that $du \in D(\Omega(A))$ for all $u \in \Omega^1(A)$.

\begin{theorem}
\label{thm:omega-otid}
The differential graded algebra $(\Omega(A),d)$ is a $\cup_1$-dga.
\end{theorem}

\begin{proof} 
The algebra $\Omega(A)$ inherits from $T(A)$ the structure of a 
graded algebra with cup-one products. Furthermore, the Steenrod 
identity \eqref{eq:steenrod-again} holds in the dga $(\Omega(A),d)$. 

We define an $R$-linear map $\circ\colon T^2(A) \otimes_R T^2(A) \to T^2(A)$ by
\begin{equation}
\label{eq:circ-omega}
(a_0 \otimes a_1 \otimes a_2) \circ (b_0 \otimes b_1 \otimes b_2)
		= a_0b_0 \otimes a_1b_1 \otimes a_2b_2\, .
\end{equation}
It is clear that this map induces a map 
$\circ\colon \Omega^2(A) \otimes_R \Omega^2(A) \to \Omega^2(A)$.
A straightforward computation using
equation \eqref{eq:ta-dga-1} for the cup product and 
equation \eqref{eq:cup-one-bat} for the cup-one product 
shows that 
$(u_1 \cup u_2)\circ (v_1 \cup v_2)
			= (u_1 \cup_1 v_1) \cup (u_2 \cup v_2)$
for all $u_1,u_2,v_1,v_2 \in \Omega^1(A)$ and thus 
the map $\circ$ restricts to a well-defined map 
$\circ\colon D^2(\Omega(A)) \otimes_R D^2(\Omega(A)) \to D^2(\Omega(A))$ 
which obeys formula \eqref{eq:circ-op}.

In view of Lemma \ref{lem:sth-cup1}, we are left with showing that 
 equation \eqref{eq:otid} holds for any elements 
$u,v,w\in \Omega^1(A)$ with $du \in D^2(\Omega(A))$. 
As in the proof of Theorem \ref{thm:cochain-cup1d}, 
we will show that 
\begin{equation}\label{eq:dc1}
u\cup_1(v \cup w) = -(u\cup_1 v) \cup w
+ \sum_i (u_{1,i}\cup_1 v) \cup (u_{2,i} \cup_1 w) 
- v \cup (u \cup_1 w)
\end{equation}
where $du = \sum_i u_{1,i} \cup u_{2,i}$.

From the definition of $\Omega(A)$, it is
sufficient to show that equation \eqref{eq:dc1} holds in $T^\ast(A)$ via the
canonical projection $J \colon T^\ast(A) \to \Omega^\ast(A)$.
For this we can identify $u, v, w$ with the following preimages under this 
projection: $u = a_0 da_1 = a_0 \otimes a_1 -a_0a_1 \otimes 1$,  
$v = b_0 \otimes b_1$, and $w = c_0 \otimes c_1$, respectively.

Then using  the formula $d(a_0da_1) = da_0 \cup da_1$ and the 
formulas for the cup product and cup-one product in $T^\ast(A)$, it follows that
\begin{align*}
u \cup_1 (v \cup w)
	& = - (a_0b_0 \otimes b_1c_0 \otimes c_1a_1)
			+	 (a_0a_1b_0 \otimes b_1c_0 \otimes c_1)\\[3 pt]
- (u \cup_1 v) \cup w
     & = -a_0b_0 \otimes a_1b_1c_0 \otimes c_1
						+ a_0a_1b_0 \otimes b_1c_0 \otimes c_1\\[3 pt]
(da_0 \cup_1 v) \cup (da_1 \cup_1 w)									
     &=  b_0 \otimes a_0b_1 c_0 \otimes a_1c_1
				- a_0b_0 \otimes b_1c_0 \otimes a_1c_1\\
		& \quad		+ a_0b_0 \otimes a_1b_1c_0 \otimes c_1
				- b_0 \otimes a_0a_1b_1c_0 \otimes c_1\\[3 pt]
-v \cup (u \cup_1 w)
		& = - b_0 \otimes a_0b_1c_0 \otimes a_1c_1 + b_0 
						\otimes a_0a_1b_1c_0 \otimes c_1.
\end{align*}
Equation  \eqref{eq:dc1} follows and the proof is complete.
\end{proof}

The following example is analogous to an example in
\cite{Abbassi-2013} to indicate that the naive right-handed 
analogue of the Hirsch identity \eqref{eq:steenrod-1c} does 
not hold in this context.

\begin{example} 
\label{ex:abbassi}
As in the proof of Theorem \ref{thm:omega-otid}, let 
$u = a_0 \otimes a_1 - a_0a_1 \otimes 1$,
$v = b_0 \otimes b_1$, and
$w = c_0 \otimes c_1$. Then
\begin{align*}
u \cup_1 (v \cup w) 
		& = - (a_0b_0 \otimes b_1c_0 \otimes c_1a_1)
			    +	 (a_0a_1b_0 \otimes b_1c_0 \otimes c_1)\\[3 pt]
(u \cup_1 v) \cup w 
		& = (a_0b_0 \otimes a_1b_1 - a_0a_1b_0 \otimes b_1)
		      \cup (c_0 \otimes c_1)\\
		& =  a_0b_0 \otimes a_1b_1 c_0 \otimes c_1
					-a_0a_1b_0 \otimes b_1c_0 \otimes c_1     \\[3 pt]
v \cup (u \cup_1 w)
		& = (b_0 \otimes b_1)\cup 
					(a_0c_0 \otimes a_1c_1 - a_0a_1c_0 \otimes c_1)\\    
		& = b_0 \otimes b_1a_0c_0 \otimes a_1c_1
					- b_0 \otimes b_1 a_0a_1c_0 \otimes c_1, 	
\end{align*}
and so 
\[
u \cup_1 (v \cup w) \ne (u\cup_1 v) \cup w + v \cup (u \cup_1 w).
\]
\end{example}

\begin{remark}
\label{rem:G-algebra}
The referee has suggested there may be a unified approach towards 
showing that  our two main examples---cochain algebras and non-commutative differential 
forms---are indeed $\cup_1$-dgas. This may rely on the notion of 
G-algebra in the sense of Kadeishvili \cite{Kadeishvili-2003} and use 
recent work of Battikh and Issaoui \cite{Battikh-2019}.
\end{remark}

\section{Binomial cup-one algebras}
\label{sect:bin-cup1}

In this section we review the notions of binomial ring and ring of 
integrally-valued polynomials, and define binomial cup-one algebras.

\subsection{Binomial rings}
\label{subsec:binomial}
Following P.~Hall \cite{Hall-1976}, Wilkerson \cite{Wilkerson},  
Elliott \cite{Elliott-2006}, and D.~Yau \cite{Yau}, we make the 
following definition.

\begin{definition}
\label{def:binomial-ring}
A commutative, unital ring $A$ is a {\em binomial ring}\/ if $A$ is torsion-free 
(as a $\Z$-module), and the elements 
\begin{equation}
\label{eq:d-a}
\binom{a}{n} \coloneqq a(a-1)\cdots (a-n+1)/n! \in A\otimes_{\Z} \Q
\end{equation}
lie in $A$ for every $a\in A$ and every $n>0$.  
\end{definition}

For instance, the ring of integers $\Z$ is a binomial ring. 

As shown by Xantcha in \cite{Xantcha},  building on work of Ekedahl 
from \cite{Ekedahl-2002}, this condition is equivalent to the 
existence of maps $\zeta_n\colon A\to A$, $a\mapsto \binom{a}{n}$ 
for all integers $n\ge 0$ such that the following five axioms are satisfied:
\begin{align}
\z_n(a+b) &= \sum_{i+j=n} \z_i(a)\z_j(b), 
\label{bin1-z}\\
\z_n(ab)   &= \sum_{m=0}^{n}\z_m(a) \cdot  \left(
\sum_{\substack{q_1+\cdots+q_m=n \\ 
q_i\geq 1}} \z_{q_1}(b)\cdots\z_{q_m}(b)\right) ,
\label{bin2-z}\\
\z_{m}(a)\z_{n}(a) &=  \sum_{k=0}^{n}\binom{m+k}{n}
\binom{n}{k} \z_{m+k}(a),
\label{bin3-z}\\
\z_n(1) &=0\quad\text{for $n\ge 2$},
\label{bin4-z}\\
\z_0(a)&=1\quad\text{and}\quad  \z_1(a)=a.
\label{bin5-z}
\end{align}

\begin{remark}
\label{rem:lambda}
As shown in \cite{Elliott-2006, Wilkerson}, 
a torsion-free ring $A$ is a binomial ring 
if and only if $A$ is a $\lambda$-ring and  all its Adams 
operations are the identity. 
\end{remark}

The following lemma shows that if $A$ is a commutative ring,
possibly without unit, that satisfies the assumptions 
of Definition \ref{def:binomial-ring},
then there is a natural way to extend $A$ to a binomial
ring by adding a unit.

\begin{lemma}
\label{lem:extend-ring}
Let $A$ be a commutative ring, possibly without unit. 
Suppose $A$ is a module over a binomial ring $R$, 
and suppose $\binom{a}{n} \in A$ for every $a\in A$ and $n>0$. 
Then $R \oplus A$ has the natural structure of a binomial ring
with unit $1\in R$, where $R$ is viewed as a direct summand of
$R \oplus A$.
\end{lemma}

\begin{proof}
For $a\in A$ and $n >0$, set $P_n(a)$ equal to
$P_n(a) = a(a-1)\cdots (a-n+1)$. Due to our assumption, 
equation \eqref{bin1-z} is equivalent to the
statement that for $n >0$, the equation 
\begin{equation}
\label{eq:Pn}
P_n(a+b) = P_n(a) 
		+ \sum_{i=1}^{n-1}\binom{n}{i}P_i(a)P_{n-i}(b)
		+ P_n(b)
\end{equation}
holds for all $a,b\in A$.

Define the multiplication on $R\oplus A$ by
setting  $(r + a)(s+b) = rs + rb + as + ab$
where $rs$ is the mutiplication in $R$, the products
$rb$ and $as$ come from the structure of $A$ as
an $R$ module, and $ab$ denotes the product in $A$. Since 
$as=sa$ and $rb=br$ for all $r,s\in R$ and $a,b\in A$, 
this defines a commutative ring structure on the 
$R$-module $R\oplus A$, with unit as advertised. 

Dividing now both sides in equation \eqref{eq:Pn}
by $n!$ and using the identity 
$\binom{n}{i} = \frac{n!}{i!(n-i)!}$, it follows that
in $(R \oplus A)\otimes_{\Z} \Q$ we have
\begin{align*}
\z_n(r  + a) = \frac{P_n(r+a)}{n!}
			& = \frac{P_n(r)}{n!} + \sum_{i+1}^{n-1} \binom{n}{i}
			      \frac{P_i(r)P_{n-i}(a)}{n!} + \frac{P_n(a)}{n!}\\
			& =  \frac{P_n(r)}{n!}  + \sum_{i+1}^{n-1} 
			      \frac{P_i(r)P_{n-i}(a)}{i!(n-i)!} +\frac{P_n(a)}{n!}\\    
			& = \z_n(r) + \sum_{i=1}^{n-1}\z_i(r)\z_{n-i}(a) 
						+ \z_n(a),    
\end{align*}
and so $\z_n(r  + a) \in A$, for every $r\in R$, $a\in A$, and $n>0$. 
Hence, $R \oplus A$ is a binomial ring.
\end{proof}

\subsection{$R$-valued polynomials}
\label{subsec:closure}
Suppose $R$ is an integral domain, and let $K=\Frac(R)$ be its 
field of fractions. Let $K[X]$ be the ring of polynomials in a 
set of formal variables $X$, with coefficients in $K$, and let $R^X$ 
be the free $R$-module on the set $X$.   
Following  \cite{CC97, Elliott-2006, Elliott-2007}, we define 
the {\em ring of $R$-valued polynomials}\/ (in the 
variables $X$ and with coefficients in $K$) as the subring 
\begin{equation}
\label{eq:bin-int}
\Int(R^X)\coloneqq \{q\in K[X] \mid q(R^X)\subseteq R\}.
\end{equation}
Assume now that the domain $R$ has characteristic $0$, that is, 
$R$ is torsion-free as an abelian group. Then $\Int(R^X)$ is a 
binomial ring, which satisfies a type of universality property that makes it into 
the {\em free binomial ring}\/ on variables in $X$.  As a consequence (at least 
when $R=\Z$), any binomial ring is a quotient of $\Int(R^X)$, for some set $X$.

As shown in \cite[Theorem 7.1]{Elliott-2006}, every torsion-free ring 
$A$ is contained in a smallest binomial ring, $\Bin(A)$, 
which is defined as the intersection of all binomial 
subrings of $A\otimes_{\Z} \Q$ containing $A$. Alternatively, 
\begin{equation}
\label{eq:bin-bin}
\Bin(A) = \Int(\Z^{X_A}) /\big( I_A \Q^{X_A} \cap \Int(\Z^{X_A})\big) ,
\end{equation}
where $\Z[X_A]$ denotes the polynomial ring in variables 
indexed by the elements of $A$, and $I_A$ is the kernel of 
the canonical epimorphism $\Z[X_A]\to A$. Moreover, if $A$ is
generated as a $\Z$-algebra by a collection of elements $\{a_i\}_{i\in J}$, 
then $\Bin(A)$ is the $\Z$-subalgebra of $A\otimes_{\Z} \Q$ 
generated by all elements of the form $\zeta_n(a_i)=\binom{a_i}{n}$ with $i\in J$ 
and $n\ge 0$.  Finally, it is readily seen that $\Bin(\Z[X])=\Int(\Z^X)$. 

\subsection{Products of binomial polynomials}
\label{subsec:prod}

Let $I\colon X\to \Z_{\ge 0}$ be a function which takes only finitely 
many non-zero values; in other words, the support of the function, 
$\supp(I) \coloneqq \{x \in X \mid I(x)\ne 0 \}$, is a finite subset of $X$. 
Given a coefficient ring $R$, we associate to such a function an 
$R$-valued polynomial function, as follows.

First let $\bx=\{x_1, \ldots, x_n\}$ be a finite subset of $X$ that contains 
$\supp(I)$. Then  the indexing function $I\colon X\to \Z_{\ge 0}$ takes 
values $I(x_k)=i_k$ for $k=1,\dots ,n$, and $0$ otherwise, and so it may 
be identified with the $n$-tuple $(i_1,\dots, i_n)\in (\Z_{\ge 0})^n$. 
We define a polynomial, $\zeta_I(\bx)$, in the variables $x_1,\dots, x_n$, by
\begin{equation}
\label{eq:zeta-x}
\zeta_I(\bx)=\prod_{k=1}^{n} \zeta_{i_k}(x_k).
\end{equation}
Clearly, $\zeta_I(\bx)$ is an $R$-valued polynomial in $\Int(R^{\bx})\subset \Int(R^X)$. 
That is to say, given an $n$-tuple $\ba=(a_1,\dots, a_n)\in R^{\bx}$, the evaluation 
$\zeta_I(\ba)\coloneqq \zeta_I(\bx)(\ba)$ of the polynomial $\zeta_I(\bx)$ at 
$x_k=a_k$ is an element of $R$. 

We now define an $R$-valued polynomial $\zeta_I\in \Int(R^X)$ by setting 
$\zeta_I=\zeta_I(\bx)$, for some finite set of variables $\bx$ with $\bx\supseteq \supp(I)$. 
Since $\zeta_0(a)=1$ for all $a\in R$, this definition is independent of the choice of $\bx$. 
Given any $\ba\in R^X$, we have a well-defined evaluation $\zeta_I(\ba)\in R$; 
in fact, we do have such an evaluation for any function $\ba\colon X\to R$, 
again since $I$ has finite support. 
For instance, if $I=\mathbf{0}$ is the function that takes only the 
value $0$, then $\zeta_{\mathbf{0}}$ is the constant polynomial 
$1$ in the variables $X$, and $\zeta_{\mathbf{0}}(\ba)=1$, 
for any $\ba\colon X\to R$.

The above notions extend to an arbitrary binomial ring $A$. 
For instance, if $I=(i_1,\dots,i_n)$ is an $n$-tuple of non-negative integers, 
the evaluation of the polynomial function $\zeta_I(\bx)$ from \eqref{eq:zeta-x} 
at an $n$-tuple $\ba=(a_1,\dots, a_n)$ of elements in $A$ is equal to 
$\zeta_I(\ba)= \prod_{k=1}^{n} \zeta_{i_k}(a_k)$. More generally, 
the evaluation $\zeta_I(\ba)\in A$ is defined for any function 
$\ba\colon X\to A$.

\subsection{A basis for integer-valued polynomials}
\label{subsec:basis}
We restrict now to the case when $R=\Z$. 
The next theorem provides a very useful $\Z$-basis for the ring $\Int(\Z^{X})$ 
of integer-valued polynomials; for a proof, we refer to \cite[Proposition XI.I.12]{CC97} 
and \cite[Lemma 2.2]{Elliott-2006}. 

\begin{theorem}
\label{thm:basis}
The $\Z$-module $\Int(\Z^{X})$ is free, with basis consisting of all 
polynomials of the form $\zeta_I$ with $I\colon X\to \Z_{\ge 0}$ 
a function with finite support.
\end{theorem}

Alternatively, one may take as a basis for $\Int(\Z^{X})$ all polynomials 
$\zeta_I(\bx)$ with $\supp(I)=\bx$, together with the constant polynomial 
$\zeta_{\bz}$. We emphasize that in the products $\zeta_{I}(\bx)$, 
there is no repetition allowed among the variables comprising the set $\bx$.  
For instance, the product $\zeta_m(x)\zeta_n(x)$ is not part of the 
aforementioned $\Z$-basis; rather, formula \eqref{bin3-z} expresses 
it as a linear combination of the binomials $\zeta_m(x),\dots,$ $\zeta_{m+n}(x)$. 
On the other hand, if $I$ and $J$ have disjoint supports, we have that 
$\zeta_{I}\cdot \zeta_{J}=\zeta_{I+J}$, and this polynomial is again 
part of the aforementioned basis for $\Int(\Z^{X})$.

\begin{example}[G.~P\'olya]
\label{ex:int-z}
Every degree $n$ integer-valued polynomial $f$ in a single variable $x$ 
can be written uniquely as a linear combination, $f(x)=\sum_{k=0}^n c_k \binom{x}{k}$, 
where the coefficients $c_k\in \Z$ are defined recursively by $c_0=f(0)$ and 
$c_k=f(k)-\sum_{i=0}^{k-1} c_i\binom{k}{i}$. 
\end{example}

As an application of the above theorem, we obtain the following universality 
property for free binomial rings.

\begin{corollary}
\label{lem:extend-map}
Let $X$ be a set, let $A$ be a binomial ring, and let 
$\phi \colon X \to A$ be a map of sets. There is then a unique 
extension of $\phi$ to a map $\tilde\phi \colon \Int(\Z^X) \to A$ 
of binomial rings.
\end{corollary}

\begin{proof}
For each finitely supported function $I\colon X\to \Z_{\ge 0}$ and 
each finite subset $\bx=\{x_1,\dots,x_n\}$ of $X$, we put 
\begin{equation}
\label{eq:tilde}
\tilde\phi \big(\zeta_I(\bx)\big) = \zeta_I(\phi(x_1),\dots, \phi(x_n)),
\end{equation}
where on the right side we are using the $\zeta$-maps of $A$. 
In view of Theorem \ref{thm:basis}, this formula defines a $\Z$-linear map, 
$\tilde\phi \colon \Int(\Z^X) \to A$.  Since $\zeta_1(x)=x$ for all $x\in X$, 
the map $\tilde\phi$ agrees with $\phi$ on $X$. 
Using equation \eqref{bin3-z}, it is readily verified that the map $\tilde\phi$ 
respects the multiplicative structures on source and target. Moreover, 
by construction, $\tilde\phi$ preserves the binomial structures. Therefore, 
$\tilde\phi$ is a morphism of binomial rings.
\end{proof}

A characterization of binomial rings in terms of integer-valued polynomials 
is given in the following theorem (see \cite[Theorem 4.1]{Elliott-2006} 
and \cite[Theorem 5.34]{Yau}).

\begin{theorem}
\label{thm:Yau}
A ring $R$ is a binomial ring if and only if the following two
conditions hold:
\begin{enumerate}
\item $R$ is $\Z$-torsion-free,
\item $R$ is the homomorphic image of a ring $\Int(\Z^X)$
of integer-valued polynomials for some set $X$.
\end{enumerate}
\end{theorem}

\begin{corollary}
\label{cor:bproduct}
Let $R_1$ and $R_2$ be binomial rings. Then the tensor product 
$R_1 \otimes_{\Z} R_2$, with product 
$(a \otimes b)\cdot (c \otimes d) = ac \otimes bd$,  
is a binomial ring.
\end{corollary}

\begin{proof}
By Theorem \ref{thm:Yau}, it suffices to show that
$R_1 \otimes_{\Z} R_2$ is $\Z$-torsion-free and the homomorphic
image of $\Int(\Z^X)$ for some set $X$.
The abelian group $R_1 \otimes_{\Z} R_2$ is torsion-free, since both 
$R_1$ and $R_2$ are torsion-free as $\Z$-modules. 
If we have ring epimorphisms $\Int(X_i)\surj R_i$ for $i=1,2$, 
then $R_1 \otimes_{\Z} R_2$ is a 
homomorphic image of $\Int(X_1) \otimes_{\Z} \Int(X_2)$. 
The result now follows, since
$\Int(X_1) \otimes_{\Z} \Int(X_2) \cong \Int(X_1 \cup X_2) $
by Theorem \ref{thm:basis}.
\end{proof}

One can also prove Corollay \ref{cor:bproduct} by considering
$R_1 \otimes_{\Z} R_2$ as a subring of $R_1 \otimes_{\Z} R_2\otimes_{\Z} \Q$.
Then the $\z_n$ maps have well-defined restrictions
to $R_1 \otimes_{\Z} R_2$.

\subsection{Binomial cup-one algebras}
\label{subsec:binom-cup1}
We now combine the notions of graded algebras with cup-one products
(\S\ref{subsec:graded-cup1}) and cup-one algebras 
(\S\ref{subsec:cupd-formula}) on one hand, 
with the notion of binomial rings (\S\ref{subsec:binomial}) 
on the other hand. 

\begin{definition}
\label{def:binomial-with-cup-one}
Let $R$ be a binomial ring. 
A graded $R$-algebra $A$ is called a 
\emph{binomial graded algebra with cup-one products}\/ 
if the following conditions are satisfied.
\begin{enumerate}[label=(\roman*)]
\item \label{cup1a-1}
$A$ is a graded algebra with cup-one products.
\item \label{cup1a-2}
The $R$-submodule $R \oplus A^1 \subset A^{\le 1}$, with multiplication 
$A^1\otimes_{R} A^1\to A^1$ given by the cup-one product, 
is a binomial ring.
\end{enumerate}
\end{definition}

The binomial ring structure from part \ref{cup1a-2} is the one 
described in Lemma \ref{lem:extend-ring}.
In concrete terms, the ring structure on $R\oplus A^1$ 
is given by $(r+a)(s+b)=rs +(rb+sa+a\cup_1 b)$, 
for $r,s\in R$ and $a,b\in A^1$. Denoting the binomial 
operations in $R$ by $\zeta_n^R$ and setting 
\begin{equation}
\label{eq:zeta-A}
\zeta_n^A(a)\coloneqq
a\cup_1 (a-1)\cup_1 \cdots \cup_1 (a-n+1)/n!
\end{equation}
for $a\in A^1$ and $n\in \N$, with the convention that $\zeta^A_0(a)=1\in R$, 
the binomial operations in  $R \oplus A^1$ 
are given by $\zeta_n(r+a)=\sum_{i+j=n} \z^R_i(r)\z^A_j(a)$.

\begin{definition}
\label{def:binomial-cup-one-algebra}
A differential graded $R$-algebra $(A,d)$ is called a 
\emph{binomial $\cup_1$-dga}\/  if it is both a $\cup_1$-dga 
and a binomial graded algebra with cup-one products.
\end{definition}

If $\varphi\colon A\to B$ is a morphism of 
binomial $\cup_1$-dgas, it follows from the definitions 
that $\varphi(\zeta_n(a)) = \zeta_n(\varphi(a))$, 
for all $n\ge 1$ and all $a\in  A^1$.

\begin{lemma}
\label{lem:3-properties}
In a binomial $\cup_1$-dga $(A,d)$, the following identities hold 
for $a,b,c\in A^1$,
\begin{align}
\label{eq:hirsch-left}
(a \cup b) \cup_1 c &= a \cup (b\cup_1 c) + (a \cup_1 c)\cup b,
\\
\label{eq:halfcup1d}
d(a \cup_1 b) &= -a \cup b - b \cup a + (da)\cup_1 b,
\\
\label{eq:zk1}
\z_{n+1}(a) &= \frac{\z_n(a)\cup_1 a - n \z_n(a)}{n+1}
\quad\text{for $n\ge 1$},
\end{align}
where in \eqref{eq:halfcup1d} we are assuming $da\in D^2(A)$ and $db=0$.
\end{lemma}
\begin{proof}
Equation \eqref{eq:hirsch-left} is the left-side Hirsch identity 
from \eqref{eq:hirsch-1c};
equation \eqref{eq:halfcup1d} is the $\cupd$ formula  \eqref{eq:c1d} 
in the case $db =0$; finally, equation \eqref{eq:zk1} follows straight 
from equation \eqref{eq:d-a}, or from equation \eqref{bin3-z} in the case 
$m=k+1$ and $n=1$. 
\end{proof}

For simplicity, we shall at times abuse notation and write formula \eqref{eq:zk1} 
as $\z_{n+1}(a) = \tfrac{1}{n+1} \big(\z_n(a)\cup_1 (a - n )\big)$.

\subsection{Differentiating the zeta maps}
\label{subsec:zeta-maps}
The goal for the rest of this section is to derive a formula for the differential of 
$\zeta_n(a)$ for $a$ a cocycle in $A^1$ and $n\ge 2$, in the setting where $A$ is a 
binomial $\cup_1$-dga. We start with the case $n=2$.

\begin{lemma}
\label{lem:good}
Let $A$ be a binomial $\cup_1$-dga.  If $a\in Z^1(A)$, then 
$d\z_2(a) = - a \cup a$. 
\end{lemma}

\begin{proof}
By assumption, we have a map $\zeta_2\colon A^1\to A^1$ given by 
$\zeta_2(a) = \tfrac{1}{2}(a \cup_1 a - a)$
for all $a\in A^1$.  If $da=0$, then $d(a \cup_1 a) = - 2 a \cup a$, 
by \eqref{eq:halfcup1d}. Hence, 
\[
d\zeta_2(a) =  \tfrac{1}{2} d(a \cup_1 a - a) = - \tfrac{1}{2} (2 a \cup a)
		= - a \cup a,
\]
and we are done.
\end{proof}

In greater generality, we have the following result relating the cup product, 
the differential, and the $\zeta$-maps in a binomial $\cup_1$-dga.

\begin{theorem}
\label{thm:cup-zeta}
Let $A$ be a binomial $\cup_1$-dga.  Then, for each $a\in Z^1(A)$ and 
each $k \ge 0$, we have
\begin{equation}
\label{eq:da-cup}
d\z_k (a) = - \sum_{\ell = 1}^{k-1}\z_\ell(a) \cup \z_{k-\ell}(a).
\end{equation}
More generally, if $I=(k_1,\dots, k_n)$ with $k_i\ge 0$ 
and $\ba=(a_1, \ldots , a_n)$ with  $a_i\in Z^1(A)$, then
\begin{equation}
\label{eq:d-zeta}
d(\z_I(\ba)) = - \sum_{\substack{I_1 + I_2=I \\ I_j\ne \bz}} 
\z_{I_1}(\ba) \cup \z_{I_2}(\ba) .
\end{equation}
\end{theorem}

\begin{proof}
The first claim is proved by induction on $k$. Since $da=0$, 
the case $k=1$ is vacuous, and the case $k=2$ is Lemma \ref{lem:good}.  
Assume now that \eqref{eq:da-cup} holds for all $a$ in $Z^1(A)$. 
Working in $\Q$-vector space $A\otimes_{\Z} \Q$ and using the formulas 
listed in Lemma \ref{lem:3-properties}, we get
\begin{align*}
d\z_{k+1} (a) 
&\underset{\eqref{eq:zk1}}{=} 
\frac{1}{k+1}d\big(\z_k(a)\cup_1 a \big) - 
\frac{k}{k+1}d \z_k(a)
\\
 &\underset{ \substack{{\eqref{eq:halfcup1d}}\\
{\eqref{eq:da-cup}}} }{=}
 -\frac{1}{k+1} \Big( \z_k(a)\cup a + a \cup \z_k(a) - d\z_k(a)\cup_1 a\Big)
 + \frac{k}{k+1}  \sum_{\ell = 1}^{k-1}\z_\ell(a) \cup \z_{k-\ell}(a)
 \\
&\underset{\eqref{eq:da-cup}}{=} 
-\frac{1}{k+1} \Big( \z_k(a)\cup a + a \cup \z_k(a) + 
  \sum_{\ell = 1}^{k-1}\big(\z_\ell(a) \cup \z_{k-\ell}(a)\big) \cup_1 a\Big)
 + \frac{k}{k+1}  \sum_{\ell = 1}^{k-1}\z_\ell(a) \cup \z_{k-\ell}(a)
 \\
&\underset{\eqref{eq:hirsch-left}}{=} 
-\frac{1}{k+1} \Big( \z_k(a)\cup a + a \cup \z_k(a) + 
  \sum_{\ell = 1}^{k-1} \big(
  \z_\ell(a) \cup ( \z_{k-\ell}(a) \cup_1 a ) + 
   (\z_\ell(a) \cup_1 a)\cup  \z_{k-\ell}(a) \big)
  \Big)
  \\
  &\hspace{1.1in}
 + \frac{k}{k+1}  \sum_{\ell = 1}^{k-1}\z_\ell(a) \cup \z_{k-\ell}(a)
 \\
   &=-\frac{1}{k+1} \Big( \z_k(a)\cup a + a \cup \z_k(a) + 
     \\
  &\hspace{1.1in}
  \sum_{\ell = 1}^{k-1} \big(
  \z_\ell(a) \cup ( \z_{k-\ell}(a) \cup_1 (a -k+\ell)) + 
   (\z_{\ell}(a) \cup_1 (a-\ell))\cup  \z_{k-\ell}(a) 
  \big)\Big)
   \\
&\underset{\eqref{eq:zk1}}{=} 
-\frac{1}{k+1} \Big( \z_k(a)\cup a + a \cup \z_k(a) + 
     \\
  &\hspace{1.1in}
  \sum_{\ell = 1}^{k-1} \big(
(k-\ell+1) \cdot  \z_\ell(a) \cup  \z_{k-\ell+1}(a) + 
(\ell +1) \z_{\ell+1}(a) \cup  \z_{k-\ell}(a) 
  \big)\Big)
   \\
   &=-\frac{1}{k+1} 
  \sum_{\ell = 1}^{k} \Big(
(k-\ell+1) \cdot  \z_\ell(a) \cup  \z_{k-\ell+1}(a) + 
\ell  \z_{\ell}(a) \cup  \z_{k-\ell+1}(a) 
\Big)
   \\
 &=- \sum_{\ell = 1}^{k}\z_\ell(a) \cup \z_{k-\ell+1}(a).
 \end{align*}
 
To prove the second claim, we need to show the following: 
Given $a_1, \ldots , a_n \in Z^1(A)$ and $k_1,\dots, k_n\in \Z_{\ge 0}$, 
we have
\begin{equation}
\label{eq:daa-cup}
\begin{split}
d\bigl(\z_{k_1}(a_1)\cup_1 \cdots \cup_1 \z_{k_n}(a_n) \bigr)
 	= \\
	&\hspace{-1.2in}
	 - \sum_{\ell_1, \ldots , \ell_n}
		\bigl( \z_{\ell_1}(a_1)\cup_1 \cdots \cup_1 \z_{\ell_n}(a_n)\bigr)
		\cup
		\bigl( \z_{k_1 -\ell_1}(a_1)\cup_1 \cdots \cup_1 \z_{k_n-\ell_n}(a_n)\bigr),
\end{split}
\end{equation}
where the sum is over all $n$-tuples, $(\ell_1, \ldots, \ell_n)$, 
of non-negative integers such that $\ell_i \le k_i$, with
the tuples $(0, \ldots , 0)$ and $(k_1, \ldots , k_n)$ excluded. 
In particular, if we set 
$a= \z_{k_1}(a_1)\cup_1 \cdots \cup_1 \z_{k_{n}}(a_{n})$
we have that $da\in D^2(A)$, as needed in order to apply 
formula \eqref{eq:halfcup1d}. 

The above claim is established by induction on $n$, the 
base case $n=1$ being \eqref{eq:da-cup}, which has 
just been verified. Assume \eqref{eq:daa-cup} holds for 
$b\coloneqq \z_{k_1}(a_1)\cup_1 \cdots \cup_1 \z_{k_{n-1}}(a_{n-1})$, 
where not all $k_i$'s are equal to $0$. 
Then, for all $a_n\in Z^1(A)$, 
\begin{align*}
d (b\cup_1 a_n) 
&\underset{\eqref{eq:halfcup1d}}{=}  -b \cup a_n - a_n \cup b + (da)\cup_1 a_n
\\
    &\underset{\eqref{eq:daa-cup}_{n-1}}{\hspace*{-10pt}=} -b \cup a_n - a_n \cup b -
    \Bigg(\sum_{\ell_1, \ldots , \ell_{n-1}}
		\bigl( \z_{\ell_1}(a_1)\cup_1 \cdots \cup_1 \z_{\ell_{n-1}}(a_{n-1})\bigr)\, \cup
\\[-5pt]
&\hspace{2.2in}
		 \bigl( \z_{k_1 -\ell_1}(a_1)\cup_1\cdots\cup_1 \z_{k_{n-1}-
		\ell_{n-1}}(a_{n-1})\bigr)\Bigg)\cup_1 a_n
\\
&\underset{\eqref{eq:hirsch-left}}{=} 	
- \z_{k_1}(a_1)\cup_1\cdots\cup_1 \z_{k_{n-1}}(a_{n-1}) \cup \z_1(a_n) - 
\z_1(a_n) \cup  \z_{k_1}(a_1)\cup_1\cdots \cup_1\z_{k_{n-1}}(a_{n-1}) -
\\
&\hspace{0.3in}
\sum_{\ell_1, \ldots , \ell_{n-1}}\! \bigg(
		\bigl( \z_{\ell_1}(a_1)\cup_1\cdots\cup_1 \z_{\ell_{n-1}}(a_{n-1})\bigr)
		\cup
		\big( \z_{k_1 -\ell_1}(a_1)\cup_1\cdots\cup_1 \z_{k_{n-1}-
		       \ell_{n-1}}(a_{n-1}) \cup_1  \z_1(a_n)\big) +\bigg.
\\
&\hspace{0.4in}
              \bigg.  \big( \z_{\ell_1}(a_1)\cup_1\cdots \cup_1\z_{\ell_{n-1}}(a_{n-1}) \cup_1 \z_1(a_n) \big) \cup 
               \big(  \z_{k_1 -\ell_1}(a_1)\cup_1\cdots\cup_1 \z_{k_{n-1}-\ell_{n-1}}(a_{n-1})\big)  \bigg)\, ,
\end{align*}
where not all $\ell_i$'s are equal to $0$. Now set $\ell=k_n$.
Rewriting this expression as a single sum yields formula \eqref{eq:daa-cup} 
in the case when $\ell=1$. The general case follows by induction on $\ell$. 
Assuming equation \eqref{eq:daa-cup} holds for $\ell=k_n$, we get 
\begin{align*}
d (\z_{k_1}(a_1)\cup_1\cdots\cup_1 \z_{k_{n}+1}(a_{n})) 
&\hspace*{6pt}=\hspace*{6pt} d(b\cup_1  \z_{\ell+1}(a_{n}))\\
&\underset{\eqref{eq:zk1}}{=} 	
d\Big(b\cup_1\big(\tfrac{1}{\ell+1} \z_{\ell}(a_n)\cup_1 a_n - 
\tfrac{\ell}{\ell+1} \z_{\ell}(a_n)\big)\Big)\\
&\hspace*{6pt}=\hspace*{6pt}
\tfrac{1}{\ell+1} d\big( b\cup_1 \z_{\ell}(a_n)\cup_1a_n\big) - 
\tfrac{\ell}{\ell+1} d\big(b\cup_1\z_{\ell}(a_n)\big)\\
&\underset{\eqref{eq:halfcup1d}}{=} 	
\tfrac{1}{\ell+1} \Big(\! - b\cup_1\z_{\ell}(a_n) \cup a_n - 
a_n \cup b\cup_1 \z_{\ell}(a_n)\, +\\[-2pt]
&\hspace*{0.8in} d\big( b\cup_1\z_{\ell}(a_n)\big)\cup_1 a_n\Big) - 
\tfrac{\ell}{\ell+1} d\big(b\cup_1 \z_{\ell}(a_n)\big).
\end{align*}
Using the induction hypothesis to express $d(b\cup_1 \z_{\ell}(a_n))$ as 
a sum according to \eqref{eq:daa-cup}, we obtain (after simplifying as 
before) equation \eqref{eq:daa-cup}, with $k_n$ replaced by $k_n + 1$.
This completes the proof.
\end{proof}

In \cite{Porter-Suciu-2021-1}, we will provide a different proof of this theorem, in 
a more general setting, by dropping the assumption that $da_i=0$ in formula 
\eqref{eq:d-zeta}.
  
\subsection{Binomial operations on cochains and forms}
\label{subsec:binomials-cochains}
We now show that cochain complexes with coefficients in a
binomial ring are binomial $\cup_1$-dgas in a natural way.

\begin{theorem}
\label{thm:cochain-bincup1d}
For any non-empty $\Delta$-complex $X$ and binomial ring $R$, the cochain 
algebra $R \oplus C^{\ge 1} (X;R)$ is a 
binomial $\cup_1$-dga.
\end{theorem}

\begin{proof}	
By Theorem \ref{thm:cochain-cup1d}, the cochain algebra 
$C^{\ge 1}(X;R)$ is a $\cup_1$-dga, with cup-one product 
on $A^1$ given by formula \eqref{eq:cup1-simplicial}. 
We define maps 
$\zeta^X_n\colon C^{1}(X;R)\to C^{1}(X;R)$ by setting 
\begin{equation}
\label{eq:zeta-cochain}
\zeta^X_n(f)(e) \coloneqq \zeta^R_n(f(e))
\quad\text{for $n \ge 1$}
\end{equation} 
for each $1$-cochain $f \in C^1(X;R)= \Hom(C_1(X;R),R)$ 
and $1$-simplex $e$ in $X$. 
Define $\z_0$ on $R \oplus C^1(X;R)$ by setting
$\z_0(f) = 1 \in R$ for all $f \in C^1(X;R)$ and 
$\z_0(r) = 1$ for all $r \in R$.
Then by Lemma \ref{lem:extend-ring} it follows that
$R \oplus C^1(X;R)$ is a binomial ring. This completes the proof.
\end{proof}

Note that the evaluation $\z_n(f)(e)$ is equal to $\binom{f(e)}{n}$; in the 
case $R=\Z$, this is simply a binomial coefficient.

\begin{remark}
\label{rem:comm-cup1-ring}
As noted in Example \ref{ex:interval}, the cup-product between $0$-cocycles 
and $1$-cocycles need not be commutative, even when $X$ is a $1$-simplex. 
This explains why in Theorem \ref{thm:cochain-bincup1d} we had 
to replace $C^0(X;R)$ by the ring $R$, so that the multiplication 
in $R\oplus C^1(X;R)$ is commutative.
\end{remark}

\begin{remark}
\label{rem:ryb-zeta}
If $A=C^*(X;\Z)$ is the cochain algebra on a $\Delta$-complex $X$, 
then the $\zeta_2$ map defined in \eqref{eq:zeta-cochain} equals minus 
the $\zeta$ map defined by Rybnikov in \cite{Rybnikov-2}. 
\end{remark}

\begin{example}
\label{ex:cochain-I-bis}
For the cochain algebra $C=C^{\ast}(I;\Z)$ from Example \ref{ex:interval}, 
the $\zeta_n$-maps are given by
$\z_n(mu) = \binom{m}{n} u$. In particular,
$\z_n(u) = 0$ for $n \ge 2$.
\end{example}

It is readily seen that the $\zeta$ maps of cochain algebras 
enjoy the following naturality property: 
If $h\colon X \to Y$ is a map of $\Delta$-complexes, 
then $\zeta^X_n(h^{\sharp}( f))=\zeta^Y_n(f)$, for all $f\in C^1(Y;R)$ 
and all $n\ge 0$.

\subsection{Binomial operations on non-commutative forms}
\label{subsec:binomials-omega}

We conclude this section by showing that non-commutative forms 
on a binomial ring are also binomial $\cup_1$-dgas.
 
\begin{theorem}
\label{thm:bin-omega}
If $A$ is a binomial ring over $R$, then 
$R \oplus \Omega^{\ge 1} (A)$ is a 
binomial $\cup_1$-dga.
\end{theorem}

\begin{proof}
Let $T(A)$ be the tensor algebra from Proposition \ref{prop:tensor-dga} 
and let $\Omega(A)\subset T(A)$ be the subalgebra with graded pieces 
defined in \eqref{eq:nc-Omega}. Recall that $\Omega(A)$ inherits from 
$T(A)$ the structure of a graded algebra with cup-one products.
Moreover, by Theorem \ref{thm:omega-otid}, the algebra 
$\Omega^\ast(A)$ is a cup-one dga (over $\Z$). 

Applying now Corollary \ref{cor:bproduct} with $R_1 = R_2 = A$, 
we obtain in a natural fashion a binomial ring structure on $A \otimes_{\Z} A$. 
Observe now that the multiplication in $A \otimes_{\Z} A$
given in the corollary coincides with the cup-one product 
defined on $T^1(A)$. Hence, the $\Z$-submodule 
$\Z\oplus T^1(A)\subset T^{\le 1}(A)$ acquires the structure  
of a commutative ring. Therefore, by 
Lemma \ref{lem:extend-ring}, 
$\Z\oplus \Omega^1(A)$
is also a binomial ring, and we are done.
\end{proof}

\section{Free cup-one algebras}
\label{sect:free-bin-dga}

In this section we define the notion of a free 
binomial cup-one dga, $\T(X)$, generated by a set, $X$.
We show that maps of binomial graded
algebras from $\T(X)$ to a binomial graded algebra $A$ are
determined by the map restricted to the elements in $X$ and
if $\T(X)$ and $A$ are binomial cup-one dgas then a map
from $\T(X)$ to $A$ commutes with the differential if and only if 
it commutes with the differential on elements in $X$. 
We will work here over the ring $R=\Z$ and abbreviate 
$\otimes=\otimes_{\Z}$. 

\subsection{The free binomial graded algebra}
\label{subsec:free-bin-alg}
Given a set $X$, we let $\T=\T(X)$ denote the free graded algebra
with $\T^1$  equal to the ideal of all polynomials without
constant term in the free 
binomial ring  $\Int(\Z^X)$.  We define a cup-one map, 
$\cup_1 \colon (\T^1 \otimes \T^1) \otimes \T^1 \to \T^2$,
by means of the left Hirsch identity \eqref{eq:hirsch-1c} and
a map $\circ \colon T^2 \otimes T^2 \to T^2$ by means
of equation \eqref{eq:circ-op}.
Then $\T(X)$ is a graded algebra with cup-one products called
the \emph{free binomial graded algebra with cup-one products
generated by $X$}.  

We now make this all more precise; we start with a definition.
Recall from section \ref{subsec:closure} that $\Int(\Z^X)$ denotes 
the free binomial ring of integer-valued polynomials in 
variables from $X$.

\begin{definition}
\label{def:free-binomial-dga}
The {\em free binomial graded algebra}\/ on a set $X$, denoted $\T=\T(X)$, 
is the tensor algebra on $\fm_X$, the ideal of all integer-valued polynomials 
in variables $X$ without constant term,
\begin{equation}
\label{eq:free-binomial-dga}
\T^{*}(X)=T^{*}(\fm_X) \, .
\end{equation} 
\end{definition}

By construction, $\T^0(X)=\Z$ and $\T^1(X)=\fm_X$, and so 
$\T^{\le 1}(X)=\T^0\oplus \T^1$ is isomorphic to the free binomial 
ring $\Int(\Z^X)$. 
By Theorem \ref{thm:basis}, $\T^1(X)$ is a free $\Z$-module, 
with basis consisting of all integer-valued polynomials of the form 
$\zeta_I=\zeta_I(\bx)$, where $I\colon X\to \Z_{\ge 0}$ has finite, non-empty 
support and $\bx=\supp(I)$. 

The $\Z$-module $\T^1$ comes endowed with a 
cup-one product map, $\T^1\otimes \T^1\to \T^1$, $a\otimes b\mapsto ab$. 
By analogy with the classical Hirsch formula for cochain algebras, we 
use this cup-one product to define a map $\T^2 \otimes \T^1\to \T^2$ by 
\begin{equation}
\label{eq:hirsch-b}
(a \otimes b)\otimes c \mapsto ac \otimes b + a \otimes bc \, . 
\end{equation}
For the terms in the $\cupd$ formula to be defined we include
the map $\circ \colon \T^2 \otimes \T^2 \to \T^2$ defined 
on basis elements by
\begin{equation}
\label{eq:circ-2}
(a_1 \otimes a_2) \circ (b_1 \otimes b_2)=
(a_1 b_1) \otimes (a_2 b_2) .
\end{equation}
With this structure,
$\T(X)$ is a graded algebra with cup-one products.

\subsection{A binomial $\cup_1$-dga structure on $\T(X)$}
\label{subsec:tx-cup1d}

The next step is to show that the free binomial graded algebra 
$\T(X)$ admits a natural structure of a binomial $\cup_1$-dga.

\begin{theorem}
\label{thm:cup1-t}
For any set $X$, the algebra $\T=\T(X)$ is a binomial $\cup_1$-dga, 
with differential $d_{\T}$ satisfying $d_{\T}(x) = 0$ for all $x \in X$.
\end{theorem}

\begin{proof}
Our goal is to define a differential $d_{\T}\colon \T\to \T$ that 
satisfies the $\cupd$ formula \eqref{eq:c1d}. We start by setting 
$d_{\T}(x)= 0$ for all $x \in X$. Next, for each basis element 
$\zeta_I=\zeta_I(\bx)$ of $\T^1(X)$ with $\supp(I) = \bx$, we set 
\begin{equation}
\label{eq:dt-bx}
d_{\T}(\zeta_I(\bx)) = - \sum_{\substack{I_1 + I_2=I \\ I_j\ne \bz}} 
\zeta_{I_1}(\bx) \ot \zeta_{I_2}(\bx) .
\end{equation}
This formula defines a $\Z$-linear map, $d_{\T}\colon \T^1 \to \T^2$. 
It remains to show that $d_{\T}$ can be extended to a 
differential on the whole of $\T$ that satisfies the $\cupd$ formula.

To achieve this goal, we first define a $2$-dimensional $\Delta$-complex, 
$\Delta(X)$, as follows. There is a single vertex; a $1$-simplex is a function 
$\ba\colon X \to \Z$; finally, to each pair of $1$-simplices, $\ba$ and $\ba'$, 
we assign a $2$-simplex, $(\ba,\ba')$, with faces $\partial_0(\ba,\ba')=\ba'$, 
 $\partial_1(\ba,\ba')=\ba+\ba'$, and  $\partial_2(\ba,\ba')=\ba$. 
We let $C(X)\coloneqq (C^{\bullet}( \Delta(X)), d_{\Delta})$ 
denote the simplicial cochain algebra of $\Delta(X)$. 

Next, we define a degree-preserving, $\Z$-linear map 
$\varphi \colon \T^{\le 2}(X) \to C(X)$,  as follows. First 
define $\varphi\colon \T^0(X) \to C^0(X)$ by sending 
$1\in \T^0(X)=\Z$ to the the unit cochain $1\in C^0(X)$. 
For each basis element $\z_I$ of $\T^1(X)$, we set 
$\varphi(\z_I)\in C^1(X)$ equal to the $1$-cochain whose value 
on a $1$-simplex $\ba$ is $\z_I(\ba)$.
Finally, we set $\varphi(\z_I \otimes \z_{J})\in C^2(X)$ equal to the 
$2$-cochain whose value on a $2$-simplex $(\ba, \ba^\prime)$ is 
$\z_{I}(\ba)\cdot \z_{J}(\ba^\prime)$. 
It follows directly from the definitions that $\varphi$ commutes with
the cup and cup-one products. Comparing the expressions for the $\circ$ maps 
given in \eqref{eq:circ-cochains} and \eqref{eq:circ-2}, we 
infer that $\varphi \circ (\circ_{\T}) = (\circ_{\Delta}) \circ (\varphi\otimes \varphi)$.
Moreover, since $d_{\Delta}\varphi(x) = 0$ for all $x\in X$, 
Theorem \ref{thm:cup-zeta} applies; comparing the expressions  
given by equations \eqref{eq:daa-cup} and \eqref{eq:d-zeta}, we 
deduce that $\varphi \circ d_{\T} =   d_{\Delta} \circ \varphi$. 

We claim that $\varphi\colon \T^{\le 2}(X) \to C(X)$
is a monomorphism. To prove the claim, first suppose that 
$\varphi\bigl( \sum \alpha_{I}\z_I \bigr) =0$, for some
constants $\alpha_I$. Then for all $1$-simplices, $\ba$,
it follows that $\sum \alpha_I\z_I(\ba)= 0$, and so 
$\sum \alpha_I\z_I$ is the zero polynomial. Since the 
polynomials $\z_I$ are elements in a basis for $\Int(\Z^X)$, 
it follows that each $\alpha_I$ is equal to $0$. Therefore, 
$\varphi\colon \T^1(X) \to C^1(X)$ is a monomorphism.

Now suppose $\varphi\bigl( \sum \alpha_{I,J} \zeta_{I}\otimes \zeta_{J}\bigr) = 0$,  
for some $\alpha_{I,J}\in \Z$. Then 
$ \sum \alpha_{I,J} \zeta_{I}(\ba)\cdot \zeta_{J}(\ba')=0$, for all functions 
$\ba, \ba'\colon X\to \Z$. Let $X'$ be another (disjoint) copy of $X$, and for each $J$, 
let $J'\colon X'\to \Z_{\ge 0}$ be the corresponding indexing function. 
Viewing each $\zeta_{J'}$ as a polynomial in $\Int(\Z^{X'})$, it follows 
that the polynomial $\sum \alpha_{I,J} \zeta_{I}\cdot\zeta_{J'} \in \Int(\Z^{X\sqcup X'})$ 
is the zero polynomial. Note that, for each pair $I$ and $J$ of indexing functions, 
the functions $I$ and $J'$ have disjoint supports; hence, $\z_{I} \cdot \z_{J'}=\zeta_{K}$, 
where $\left.K\right|_X=I$ and $\left.K\right|_{X'}=J'$. 
Since these polynomials are elements in a basis for $\Int(\Z^{X\sqcup X'})$, 
it follows that each $\alpha_{I,J}$ is equal to $0$, thus showing that 
$\varphi \colon \T^2(X) \to C^2(X)$ is a monomorphism.

This completes the argument that the map 
$\varphi\colon \T^{\le 2}(X) \to C(X)$ is a monomorphism that
commutes with the cup products, the cup-one products, 
the $\circ$ maps, and the respective $d$ maps. 
By Theorem \ref{thm:cochain-bincup1d}, the differential 
$d_{\Delta}$ satisfies the $\cupd$ formula, and so it follows that 
$d_{\T}$ also satisfies the $\cupd$ formula \eqref{eq:c1d}.

Next, we extend $d_{\T}$ to $\T=\T(X)$ using the
graded Leibniz rule. To complete the proof, we need 
to show that $d_{\T} \circ d_{\T}=0$. It suffices to 
show this for the map $d_{\T} \circ d_{\T} \colon \T^1 \to \T^3$.
By applying $d_{\T}$ to the first factors in the sum 
of tensor products from the right-hand side of 
\eqref{eq:dt-bx}, we have that
\[
\sum d_{\T}(\z_{I_1}(\bx)) \ot \z_{I_2}(\bx)
= \sum \z_{J_1}(\bx) \ot \z_{J_2}(\bx) \ot \z_{J_3}(\bx),
\]
where the sum is over all finitely-supported functions 
$J_1,J_2,J_3\colon X\to \Z_{\ge 0}$ with $J_i \ne \bz$ 
and $J_1 + J_2 + J_3=I$. Similarly,
\[
\sum \z_{I_1}(\bx) \ot d_{\T}(\z_{I_2}(\bx))
= \sum \z_{J_1}(\bx) \ot \z_{J_2}(\bx) \ot \z_{J_3}(\bx),
\]
and using the graded Leibniz rule we have
\begin{align*}
d_{\T}\circ d_{\T}(\z_I(\bx))
		& = - \sum d_{\T}(\z_{I_1}(\bx)) \ot \z_{I_2}(\bx)
					+ \sum \z_{I_1}(\bx) \ot d_{\T}(\z_{I_2}(\bx))\\
		& = \sum 	- \z_{J_1}(\bx) \ot \z_{J_2}(\bx) \ot \z_{J_3}(\bx)
					+ \z_{J_1}(\bx) \ot \z_{J_2}(\bx) \ot \z_{J_3}(\bx)\\
		& = 0.	
\end{align*}
This completes the proof.
\end{proof}

\begin{remark}
\label{rem:non-zero-d}
Examples of $\cupd$ dgas, $(\T(X),d)$ for which $d(x)$ is not zero
for all $x\in X$ along with applications will be given in 
\cite{Porter-Suciu-2021-1}.
\end{remark}

\subsection{Maps with domain $\T(X)$}
\label{subsec:tx-maps}

We will show in \cite{Porter-Suciu-2021-1} how to enhance the structure  
of the free binomial graded algebra on a set $X$ so as to construct a 
$1$-minimal model over the integers for an arbitrary binomial $\cup_1$-dga. 

As a stepping stone towards that goal, we show in the remainder of this 
section that maps from $\T(X)$ to a binomial graded algebra with 
cup-one products, $A$, are determined by maps of sets 
from $X$ to $A^1$ and in the case that $\T(X)$ and $A$ 
are binomial $\cup_1$-dgas, then a map of binomial 
$\cup_1$-dgas commutes with the respective differentials 
if and only if the map commutes with the differentials of the 
generators $x \in X$. We start with an extension lemma.

\begin{lemma}
\label{lem:TtoA}
Let $X$ be a set, let $A$ be a binomial graded $\Z$-algebra with 
cup-one products, and let $\phi \colon X \to A^1$ be a map of sets. 
There is then a unique extension of $\phi$ to a map $f \colon \T(X) \to A$ 
of binomial graded algebras.
\end{lemma}

\begin{proof}
Recall that $\T^{\le 1}(X)$ is a graded ring, with $\T^0(X)=\Z$ the 
constant polynomials, and $\T^1(X)$ all integer-valued polynomials 
in variables $X$ with zero constant term.  
In view of Definition \ref{def:binomial-with-cup-one}, 
the $\Z$-submodule $\Z\oplus A^1 \subset A^{\le 1}$, with multiplication 
$A^1\otimes A^1\to A^1$ given by the cup-one product, is a binomial ring.

The proof of Corollary \ref{lem:extend-map} shows that the set map 
$\phi \colon X \to A^1$ extends to a map of binomial rings, 
$\tilde\phi\colon \Z\oplus \T^1(X)\to \Z\oplus A^1$ which is the identity 
in degree $0$.  Finally, since $\T(X)$ is the free graded algebra 
generated by $\T^1(X)$, the restriction of $\tilde\phi$ to degree $1$ 
pieces, $f \colon \T^1(X) \to A^1$, extends uniquely to a map 
$f \colon \T(X) \to A$ of binomial graded algebras with cup-one products.
\end{proof}

\begin{theorem}
\label{thm:tx-a}
Let $X$ be a set, let $(\T(X), d_{\T})$ and $(A,d_A)$ be
binomial $\cup_1$-dgas, and let $f\colon \T(X) \to A$ be a map
of graded algebras with cup-one products. Then $f$ 
commutes with the respective differentials if and only if 
$d_A \circ f(x) = f \circ d_{\T}(x)$ for all $x \in X$.
\end{theorem}

\begin{proof}
From the $\cupd$ formula \eqref{eq:c1d}, equation \eqref{eq:zk1},
and the left Hirsch identity \eqref{eq:hirsch-1c}, it follows that if
$d_A \circ f(x) = f \circ d_{\T}(x)$ for all $x \in X$, then
$d_A \circ f(a) = f \circ d_{\T}(a)$ for all $a \in \T^1(X)$.
The result now follows, since $\T^n$, for $n \ge 2$,
is generated by products of elements in $\T^1$, and 
both $d_{\T}$ and $d_A$ satisfy the graded Leibniz rule.
\end{proof}

The next corollary follows at once from Theorem \ref{thm:tx-a}.

\begin{corollary}
\label{cor:freebindga}
Let $(\T(X), d)$ be the free binomial $\cup_1$-dga on a set $X$, 
with $d(x) = 0$ for all $x \in X$.  If $A$ is a binomial $\cup_1$-dga, 
then there is a bijection between binomial $\cup_1$-dga maps from 
$(\T(X), d)$ to $(A, d_A)$ and set maps from $X$ to $Z^1(A)$.
\end{corollary}

\section{$\Z_p$-binomial rings and binomial $\cup_1$-dgas}
\label{sect:zp-binomial}

The purpose of this section is to extend the results in the
previous sections to binomial rings and $\cup_1$-dgas 
over the prime field of characteristic $p>0$.

\subsection{Definition and properties of $\Z_p$-binomial algebras}
\label{subsec:zp-binomial algebras}

Fix a prime $p$, and let $\Z_p=\Z/p\Z$ be the field with $p$ elements.
Let $A$ be a commutative $\Z_p$-algebra; we will assume that the structure 
map $\Z_p\to A$ which sends $1\in \Z_p$ to the identity $1\in A$ is injective. 
Note that the binomial operations $\z_n(a) = a(a-1)\cdots(a-n+1)/n!$ with $a\in A$ 
are defined for $1 \le n \le p-1$, since $n!$ is then a unit in $\Z_p$.

\begin{example}
\label{ex:cochains-zp}
Let $A = C^*(X;\Z_p)$ be the cochain algebra 
of a $\Delta$-complex $X$ over $\Z_p$. For a cochain $a \in A^1$, 
we have that $a(a-1)\cdots(a-n+1)= 0$ for $n \ge p$, where the 
product is the $\cup_1$-product on $A^1$. 
To see this, let $e$ be any $1$-simplex in $X$; then the elements
$a(e), a(e)-1, \ldots , a(e)-p+1$ are distinct elements in
$\Z_p$. Since there are $p$ of these elements, one of the
elements must be $0$ and the property follows.
\end{example}

This motivates the following definition.

\begin{definition}
\label{def:zp-binomial algebra}
Let $A$ be a commutative $\Z_p$-algebra. We say that $A$ is a 
\emph{$\Z_p$-binomial algebra}\/ if 
$a(a-1)\cdots(a-n+1)= 0$ for all integers $n \ge p$ and all $a \in A$.
\end{definition}

The next step is to derive properties of binomials in a 
$\Z_p$-binomial algebra analogous to those for a binomial
ring over $\Z$. We start by defining the analogue of $\Int(\Z^X)$.

Given a set $X$, we will denote by $\Int(\Z_p^X)$ the quotient of
the free binomial algebra $\Int(\Z^X)$ by the ideal generated 
by the elements $\z_n(x)$ for $x \in X$ and $n \ge p$, tensored with
$\Z_p$. We next show that 
$\Int(\Z_p^X)$ has $\Z_p$-basis given by products of the elements
$\z_i(x)$ for $0< i <p$ and $x \in X$. Recall from \eqref{eq:zeta-x} 
that, for a finite subset $\bx=\{x_1,\dots, x_n\}\subset X$ and a 
finitely supported function $I\colon X\to \Z_{\ge 0}$, we write 
$\zeta_I(\bx)=\prod_{k=1}^{n}\zeta_{I(x_k)}(x_k)$.

\begin{lemma}
\label{lem:basis-p} 
The ring $\Int(\Z_p^X)$ is a $\Z_p$-binomial algebra, with $\Z_p$-basis 
given by the $\Z_p$-valued polynomials $\zeta_I(\bx)$ with 
$I\colon X\to \{0,\dots, p-1\}$.
\end{lemma}

\begin{proof}
To show that $\Int(\Z_p^X)$ is a $\Z_p$-binomial algebra,
note that  $a(a-1)\cdots(a-n+1) = n! \cdot \z_n(a)$ for 
$a \in \Int(\Z^X)$. For $n \ge p$, we have that $n!$ is 
a multiple of $p$. Hence, for $n \ge p$ the image of 
 $a(a-1)\cdots(a-n+1)$ is the zero element in $\Int(\Z_p^X)$.
Since the projection map $\Int(\Z^X)\to \Int(\Z_p^X)$ is an
epimorphism of (graded) rings, it follows that $\Int(\Z_p^X)$
is a $\Z_p$-binomial algebra. 

To find a basis for $\Int(\Z_p^X)$ note 
that from equation \eqref{bin3-z} it follows that
$\z_i(x)\z_j(x)$ is a linear combination of elements of the form 
$\z_\ell(x)$ with $\max\{i,j\} \le \ell \le i+j$.
Hence the ideal of $\Int(\Z^X)$ generated by the elements 
$\z_i(x)$ with $i \ge p$ is the $\Z_p$-subspace generated by 
sums of polynomials of the form 
$\prod_{x_i \in X}\z_{j_i}(x_i)$ with $j_i \ge p$ for at least
one $j_i$. The result now follows from 
the integral basis theorem, Theorem \ref{thm:basis}.
\end{proof}

\begin{theorem}
\label{thm:universal-p}
Let $A$ be a $\Z_p$-binomial algebra. There is then a  
bijection between maps of    
$\Z_p$-binomial algebras from $\Int(\Z_p^X)$ to
$A$ and set maps from $X$ to $A$.
\end{theorem}

\begin{proof}
From the definition of $\Int(\Z_p^X)$ and Lemma \ref{lem:basis-p}, it follows that
$\Int(\Z_p^X)$ is the ring of finite sums of products of integer powers of the variables 
$x \in X$ with coefficients in $\Z_p$ without constant term, modulo 
the ideal generated by products of the form $a(a-1)\cdots (a- p + 1)$.
Hence, a map of sets $\phi \colon X \to A$ extends uniquely
to a multiplicative map, $\tilde\phi\colon \Int(\Z_p^X) \to A$.
Since $n!$ is a unit in $\Z_p$ for $0 < n <p$, it follows
that $\tilde\phi$ commutes with the zeta maps, and the proof
is complete.
\end{proof}

\begin{lemma}
\label{lem:binom-p}
The equations \eqref{bin1-z} through \eqref{bin5-z} hold in $\Int(\Z_p^X)$, 
where $\z_i(a)$ is defined only for $0< i <p$ and the binomials in \eqref{bin3-z} are 
reduced mod $p$.
\end{lemma}

\begin{proof}
The result follows since the projection $\Int(\Z^X) \to
\Int(\Z_p^X)$ is a map of rings that commutes with the
$\z_i$ maps for $0< i <p$.
\end{proof}

\begin{corollary}
\label{lem:binom-alg-p}
For a $\Z_p$-binomial algebra $A$, equations \eqref{bin1-z} 
through \eqref{bin5-z} hold in $A$, where $\z_i(a)$ is defined 
only for $0< i <p$ and the binomials in \eqref{bin3-z} are 
reduced mod $p$.
\end{corollary}

\begin{proof}
This is an immediate consequence of Theorem 
\ref{thm:universal-p} and Lemma \ref{lem:binom-p}.
\end{proof}

\subsection{$\Z_p$-Binomial $\cup_1$-dgas}
\label{subsec:zp-bin-cup1}

We now adjust the notion introduced in 
Definition \ref{def:binomial-cup-one-algebra} 
to fit this context.

\begin{definition}
\label{def:zp-bin-cup1-alg}
A differential graded algebra, $(A,d)$, over $\Z_p$ is called a
\emph{$\Z_p$-binomial $\cup_1$-dga}\/ if the following conditions 
are satisfied.
\begin{enumerate}
\item \label{zp1} 
$A$ is a graded $\Z_p$-algebra with cup-one products.
\item \label{zp2} 
The differential $d$ satisfies the $\cupd$ formula \eqref{eq:c1d}.
\item \label{zp3} 
The $\Z_p$-vector subspace 
$\Z_p \oplus A^1\subset A^{\le 1}$, with multiplication on $A^1$ given
by the cup-one product, is a $\Z_p$-binomial algebra.
\end{enumerate}
\end{definition}

\begin{example}
\label{ex:cochain-cup1-zp}
Let $X$ be a $\Delta$-complex. 
Using the result in Example \ref{ex:cochains-zp}, 
it is readily verified 
that the cochain algebra $\Z_p \oplus C^{\ge 1}(X;\Z_p)$ 
is a $\Z_p$-binomial $\cup_1$-dga, where
the $\z$ maps,
$\zeta^X_n\colon C^{1}(X;\Z_p)\to C^{1}(X;\Z_p)$, 
are defined by setting 
\begin{equation*}
\zeta^X_n(f)(e) \coloneqq 
\frac{f(e)(f(e)-1)\cdots(f(e)-n+1)}{n!}
\end{equation*} 
for each integer $1 \le n \le p-1$, 
for each $1$-cochain $f \in C^1(X;\Z_p)= \Hom(C_1(X;\Z_p),\Z_p)$, 
and for each $1$-simplex $e$ in $X$. 
\end{example}

\begin{theorem}
\label{thm:cup-zeta-p}
Let $A$ be a $\Z_p$-binomial $\cup_1$-dga.  
Then, for each $a\in Z^1(A)$ and 
each integer $k$ with $1 \le k \le p-1$, we have
\begin{equation}
\label{eq:da-cup-p}
d\z_k (a) = - \sum_{\ell = 1}^{k-1}\z_\ell(a) \cup \z_{k-\ell}(a).
\end{equation}
More generally, if $I=(k_1,\dots, k_n)$ with $1\le k_i\le p-1$ 
and $\ba=(a_1, \ldots , a_n)$ with  $a_i\in Z^1(A)$, then
\begin{equation}
\label{eq:daa-cup-p}
d(\z_I(\ba)) = - \sum_{\substack{I_1 + I_2=I \\ I_j\ne \bz}} 
\z_{I_1}(\ba) \cup \z_{I_2}(\ba) .
\end{equation}
\end{theorem}

\begin{proof}
The proof follows the same steps as in the proof of 
Theorem \ref{thm:cup-zeta} for $1 \le k \le p-1$.
\end{proof}

\begin{remark}
\label{rem:p=2}
For $p=2$, equation \eqref{eq:da-cup-p} is vacuously true, 
since $d\zeta_1(a)=0$ by assumption. 
On the other hand, equation \eqref{eq:daa-cup-p} is true, 
but not tautologically so. For instance, as a consequence 
of formula \eqref{eq:halfcup1d}, the identity 
$d(\z_1(a_1)\z_1(a_2)) = - a_1 \cup a_2 - a_2 \cup a_1$ 
holds over $\Z_2$.
\end{remark}

\subsection{The free $\Z_p$-binomial graded algebra}
\label{subsec:free-zp}
In this section we define the free $\Z_p$-bino\-mial graded algebra,
$\T(X;\Z_p)$, generated by a set $X$ and show
that $\T(X;\Z_p)$ has properties analogous to those of
$\T(X)$. For the rest of this section, we will abbreviate 
$\otimes_{\Z_p}$ by $\otimes$.

\begin{definition}
\label{def:Tzp}
For a set $X$ and prime $p$ the 
\emph{free $\Z_p$-binomial graded algebra}, denoted 
$\T(X;\Z_p)$, is the free non-commutative algebra 
over $\Z_p$ generated by $\Int(\Z_p^X)$.
\end{definition}

Let $\T=\T(X;\Z_p)$.  The $\Z_p$-vector space $\T^1$ comes endowed 
with a cup-one product map, $\T^1 \otimes \T^1 \to \T^1$,
$a \otimes b \mapsto ab=a \cup_1 b$.
As in section \ref{subsec:free-bin-alg}, 
we use the cup-one product to define 
a map $\T^2 \otimes \T^1 \to \T^2$ by
\begin{equation}
\label{eq:abc-otimes}
(a\otimes b) \otimes c \mapsto ac \otimes b + a \otimes bc.
\end{equation}
For the terms in the $\cupd$ formula to be defined, we define
the map $\circ\colon \T^2 \otimes \T^2 \to \T^2$ by
\begin{equation}
\label{eq:a1a2-cup}
(a_1 \cup a_2) \otimes (b_1 \cup b_2) \mapsto
(a_1 \cup_1 b_1) \cup (a_2 \cup_1 b_2).
\end{equation}
With this structure, $\T(X;\Z_p)$ is a graded $\Z_p$-algebra with
cup-one products.

The following results are analogous to those in 
section \ref{sect:free-bin-dga}. Moreover, in each case the
proof follows the same steps as for the corresponding result
over $\Z$.

\begin{theorem}
\label{thm:cup1-t-zp}
For any set $X$, the algebra $\T=\T(X;\Z_p)$ is a $\Z_p$-binomial $\cup_1$-dga, 
with differential $d_{\T}$ satisfying $d_{\T}(x) = 0$ for all $x \in X$.
\end{theorem}

\begin{lemma}
\label{lem:TtoA-p}
Let $X$ be a set, let $A$ be a $\Z_p$-binomial graded algebra, 
and let $\phi \colon X \to A^1$ be a map of sets. 
There is then a unique extension of $\phi$ to a map 
$f \colon \T(X;\Z_p) \to A$ of $\Z_p$-binomial graded algebras.
\end{lemma}

\begin{theorem}
\label{thm:tx-a-p}
Let $(\T(X;\Z_p), d_{\T})$ be the free $\Z_p$-binomial 
$\cup_1$-dga on a set $X$, let $(A,d_A)$ 
be a $\Z_p$-binomial $\cup_1$-dga, 
and let $f\colon \T(X;\Z_p) \to A$ be a map
of graded algebras over $\Z_p$ with cup-one products. 
Then $f$ commutes with the respective differentials if and only if 
$d_A \circ f(x) = f \circ d_{\T}(x)$ for all $x \in X$.
\end{theorem}

\begin{corollary}
\label{cor:freebindga-p}
Let $(\T(X;\Z_p), d)$ be the free $\Z_p$-binomial 
$\cup_1$-dga on a set $X$, 
with $d(x) = 0$ for all $x \in X$.  
If $A$ is a $\Z_p$-binomial $\cup_1$-dga, 
then there is a bijection between 
$\Z_p$-binomial $\cup_1$-dga maps from 
$(\T(X;\Z_p), d)$ to $(A, d_A)$ and set maps from 
$X$ to $Z^1(A)$.
\end{corollary}

\section{Massey  products in binomial $\cup_1$-dgas}
\label{sect:massey}

In this section we outline a relationship between the binomial operations 
and Massey products in a binomial $\cup_1$-dga. 

\subsection{Relating the $\z_i$ maps to Massey products}
\label{subsec:mtp-zeta}
We start with Massey products of the form $\langle u, \dots , u\rangle$, 
where $u$ is an (integral) cohomology class in degree $1$.   
We will develop this idea, in a more general context, 
in \cite{Porter-Suciu-2021-3}.

\begin{proposition}
\label{prop:mtp}
Let $A$ be a binomial $\cup_1$-dga over $\Z$ and let $u$ be any
element in $H^1(A)$. For each  integer $n\ge 3$, the $n$-tuple Massey 
product   $\langle u, \ldots , u \rangle$ is defined and contains $0$. 
\end{proposition}

\begin{proof}
Let $a \in A^1$ be a cocycle with $[a] = u$. 
We first treat the case $n=3$. Since 
$d\z_2(a) = - \z_1(a) \cup \z_1(a)$, it follows that
$\langle u, u, u\rangle$ is defined and contains 
$[-\z_1(a)\cup \z_2(a) - \z_2(a) \cup \z_1(a)]$.
Since $d\z_3(a) = -\z_1(a)\cup \z_2(a) - \z_2(a) \cup \z_1(a)$,
it follows that $\langle u,u,u\rangle$ contains $0$.
 
In the general case, it follows from equation \eqref{eq:da-cup} with 
$1 \le k < n$ that $(-1)^{n}\sum_{k=1}^{n}\z_k(a) \cup \z_{k-n}(a)$ is a 
cocycle with cohomology class in the $n$-tuple Massey product
$\langle [a], \ldots , [a] \rangle$. Then from equation
\eqref{eq:da-cup} with $k=n$ it follows that this cohomology
class is zero. Thus, for any element $u \in H^1(A)$, 
the $\z_k$ maps show that $\langle u, \ldots , u \rangle$ 
contains $0$.
\end{proof}
 
If the notion of binomial $\cup_1$-dga is replaced by an arbitrary 
$\cup_1$-dga, then these $n$-tuple Massey products are not 
necessarily defined, and if defined, may not contain $0$.  
 
 \begin{example}
 \label{ex:Massey-undefined-bis}
 Let $A$ be the subalgebra of $\T(\{x\})$ generated in degree
 one by powers of $x$. Following the steps in the proof of
 equation \eqref{eq:da-cup} with 
 $\binom{n}{k-1} + \binom{n}{k} = \binom{n+1}{k}$ in place of
 equation \eqref{eq:zk1}, it follows that
$dx^n = - \sum_{k=1}^{n-1} \binom{n}{k}x^k \cup x^{n-k}$.  
Therefore, $A$ is a $\cup_1$-dga. Moreover, 
$[x] \cup [x] \ne 0$ in $H^2(A)$, and so the Massey triple 
product $\langle [x], [x], [x] \rangle$ is not defined.
 \end{example}
 
Thus, over the integers the Massey triple product
$\langle u,u,u \rangle$ for $u\in H^1(A)$ may not be defined
for $A$ a $\cup_1$-dga, but is always defined and contains 
zero if $A$ is a binomial $\cup_1$-dga.

\subsection{Massey triple products in characteristic $3$}
\label{subsec:Massey-3}
As we saw in Proposition \ref{prop:mtp}, if $A$ is a binomial $\cup_1$-dga 
(over $\Z$) and if $u$ is any element in $H^1(A)$,  
then the Massey product $\langle u, u, u \rangle$ is defined 
and contains $0$.  If $A$ is defined over $\Z_3$, though, such 
a Massey product need not vanish anymore, 
due to the lack of a $\zeta_3$ map in this context; 
see Remark \ref{rem:mtp-3-ne0}.

In this section, we use a graphical approach (based on Figure \ref{fig:mtp-3}) 
to analyze this phenomenon in more detail in the case when 
$A= C^\ast(X;\Z_3)$ is the cochain algebra of a $\Delta$-complex $X$, 
with coefficients in $\Z_3$. In the next section, we will use a different 
approach to study $p$-fold Massey products in $C^\ast(X;\Z_p)$ 
and generalize the next result.

Let $\delta \colon A^i\to A^{i+1}$ be 
the Bockstein operator associated to the coefficient sequence 
$0\to \Z_3\to\Z_9\to\Z_3\to 0$.  The following proposition shows that 
the element $[- a\cup \z_2(a) - \z_2(a) \cup a] \in 
\langle [a], [a], [a] \rangle$ is the negative of the mod $3$ Bockstein
applied to $[a]$, and hence in general, is nonzero. 

\begin{figure}
\begin{tikzpicture}[scale=0.95]
\draw (-4,0) -- (4,0) -- (0,6.9282) -- (-4,0);
\draw (4,0) -- (-2, 3.4641);
\draw (-4,0) -- (2, 3.4641);
\draw (0,0) -- (0,6.9282);
\draw [red, thick] (-1, 5.196152) -- (-.5, 4.6188);
\draw [<-, red, thick] (-1, 4.6188) -- (-.5, 4.943);
\node [red] at (-1.05, 4.4188) {\footnotesize $a$};
\node[red] at (-.5, 5.125) {\footnotesize$1$};
\draw[->,red,thick] (.3,4.943) -- (.3,4.143);
\node [red]  at (.3,5.15) {\footnotesize $1$};
\node [red] at (.3, 4) {\footnotesize $a$};
\node at (-4.25,0) {\footnotesize$0$};
\node at (4.25,0) {\footnotesize$1$};
\node at (0, 7.1) {\footnotesize$2$};
\node at (-2.25, 3.4641) {\footnotesize$5$};
\node at (2.25, 3.4641) {\footnotesize$4$};
\node at (0, -.25) {\footnotesize$3$};
\node at (-.15,2.55094) {\footnotesize$6$};
\draw [red, thick] (-0.5, 4.6188) -- (0.5,4.6188); 
\draw [red, thick]  (0.5, 4.6188) --  (2.165, 1.7349);
%
\draw [red, thick] (2.165, 1.7349) -- (3, 1.73205);
\draw [->, red, thick] (2.65, 1.5) -- (2.3, 2.1062);
\node [red, thick] at (2.75,1.4) {\footnotesize$1$};
\node [red,thick] at (2.5, 2.2) {\footnotesize$a$};
%
\draw [->, red, thick]  (1.9, 2.799) -- (1.3,2.4526);
\node [red, thick] at (2, 2.8) {\footnotesize$1$};
\node [red, thick] at (1.15, 2.45) {\footnotesize$a$};
%
\draw [->, red, thick] (2.3,1.4) -- (1.8, 1.6887);
\node [red, thick] at (2.3,1.23) {\footnotesize$2$};
\node [red, thick] at (1.65,1.65) {\footnotesize$a$};
%
\draw [red, thick] (2.165, 1.7349) -- (1.5,.6);
\draw [red, thick] (-1.5,.6) -- (-1.5,0);
%
%
\draw [red, thick] (-1.5, .6) --  (1.5,.6);
\draw[->,red,thick] (.8, .4) -- (.8, .8);
\node [red] at (.8,.2) {\footnotesize $2$};
\node [red] at (.8,.9) {\footnotesize $a$};
%
\draw [->, red, thick] (-1.8,.3) -- (-1.2, .3);
\node [red, thick] at (-1.89,.3) {\footnotesize$1$};
\node [red, thick] at (-1,.3) {\footnotesize$a$};
%
\draw [blue, thick] (-1.5, 1.2) -- (1,1.2) -- (1.5,1.825);
\draw[->, blue, thick] (-.8,1) -- (-.8, 1.4);
\node [blue,thick] at (-.8,1.55) {\footnotesize$\z_2(a)$};
\node [blue,thick] at (-.71,.9) {\footnotesize$1$};
%
\draw [->] (-4,-1) -- (0,-1);
\draw (0,-1) -- (4,-1);
\node at (-.23, -.8) {\footnotesize$x$};
\draw [->] (5,0) -- (3,3.464);
\draw (3, 3.464) -- (1, 6.9282);
\node at (3.3, 3.3) {\footnotesize$x$};
\draw [->] (-1, 6.9282) -- (-3, 3.464);
\draw (-3, 3.464) -- (-5,0);
\node at (-3.3,3.3) {\footnotesize$x$};
\end{tikzpicture}
\caption{Cochains for the Massey triple product 
$\langle [a], [a], [a] \rangle$ in $K(\Z_3,1)$}
\label{fig:mtp-3}
\end{figure}

\begin{proposition}
\label{prop:mtp-3}
Let $X$ be a $\Delta$-complex, and let 
$u \in H^1(X;\Z_3)$; then the Massey triple product
$\la u,u,u\ra$ contains the element $-\delta(u)$.
\end{proposition}

\begin{proof}
Since every element in $H^1(X;\Z_3)$ is represented by a 
map from $X$ to an Eilenberg--MacLane space $K(\Z_3,1)$, 
it follows from equation \eqref{eq:massey-func} and the
naturality of the Bockstein, that it is sufficient to prove the result for 
$u$ a generator of $H^1(K(\Z_3,1);\Z_3)\cong \Z_3$.
For this it is enough to show that the result holds for the $2$-skeleton,
$Y$, of a CW-complex structure for $K(\Z_3,1)$.

Such a space is pictured in Figure \ref{fig:mtp-3}, which gives a 
$\Delta$-complex, $Y$, for the $2$-dimensional CW-complex obtained by
attaching a $2$-cell to a circle by an attaching map of degree $3$. 
As in Example \ref{ex:torus-cup}, line segments transverse to the $1$-cells in a 
$2$-dimensional $\Delta$-complex are used to picture $1$-cochains.
The numbers on the arrows give the values of the cochain on the
$1$-cells in the $\Delta$-complex.  Note that with 
$\Z_3$-coefficients, the $1$-cochain $a$ indicated 
in the figure is a cocycle and $[a]$ is a generator
of $H^1(Y;\Z_3)$.

The final step is to show that the 
cohomology class of the cocycle 
$-a \cup \z_2(a) - \z_2(a) \cup a$
in Figure \ref{fig:mtp-3} is $-\delta([a])$.
Note that the support of $-a \cup \z_2(a) - \z_2(a) \cup a$
is the $2$-simplex $[0,3,6]$, and the value of this cochain on
$[0,3,6]$ is $-1$. Then viewing the arrows with numbers as 
cochains with integer coefficients, it follows that the support
of $d\hat{a}$ is also $[0,3,6]$ and 
the value of $d\hat{a}$ on $[0,3,6]$
is $3$, where $\hat{a}$ denotes the integer
cochain indicated by the arrows labelled $a$.
This completes the proof.
\end{proof}

\begin{remark}
\label{rem:mtp-3-ne0}
Note that in the proof of Proposition \ref{prop:mtp-3}, the
element $u = [a]$ is a generator for 
$H^1(K(\Z_3,1);\Z_3) \cong \Z_3$ and $u\cup u = 0$, since
$d\z_2(a) = - a \cup a$. 
Recall from \S\ref{subsec:massey-prod}
that the indeterminacy of a Massey triple
product $\la u_1, u_2, u_3 \ra$ of cohomology classes
in $H^1(A)$, where $A$ is a dga, is the set of elements
in $u_1 \cup H^1(A) + H^1(A) \cup u_3$.
In this example, where $A=C^{\ast}(K(\Z_3,1),\Z_3)$, the cup product map
$H^1 \otimes H^1\to H^2$ is the zero map, so the indeterminacy of
$\la u,u,u\ra$ is zero. It then follows from 
Proposition \ref{prop:mtp-3} that $-\delta(u)$ is the only 
element in $\la u,u,u\ra$. Since $-\delta(u) \ne 0$ in
$H^2(K(\Z_3,1);\Z_3)$, 
this gives an example of a 
Massey product of the form $\la u,u,u\ra$ that does not vanish.
\end{remark}

\subsection{$p$-fold Massey products in characteristic $p$}
\label{subsec:Massey-p}
The next result generalizes Proposition \ref{prop:mtp-3}, from triple 
Massey products in characteristic $3$ to  $p$-fold Massey products 
in characteristic an odd prime $p$.  For a $\Delta$-complex $X$, we 
let $\delta_p\colon H^1(X;\Z_p) \to H^2(X;\Z_p)$ be the Bockstein 
operator associated to the coefficient sequence 
$0\to \Z_p\to\Z_{p^2}\to\Z_p\to 0$. 

\begin{theorem}
\label{thm:mod-p-bock}
For each $u \in H^1(X;\Z_p)$ with $p$ 
an odd prime, the $p$-fold Massey product 
$\la u,\dots ,u\ra$ is defined and contains the 
element $-\delta_p(u)$.
\end{theorem}

\begin{proof}
Theorem 2 in \cite{Porter} gives a general
formula for Massey products of $1$-dimensional
cohomology classes in a finite CW-complex whose 
$1$-skeleton is a wedge of circles. The formula is
in terms of the coefficients in the Magnus expansions
of the words corresponding to the attaching maps of the
$2$-cells. 

First let $Y$ be the $2$-dimensional complex
obtained by attaching a $2$-cell to a circle by an attaching map
of degree $p$, and let $v\in H^1(Y;\Z_p)$ be the cohomology class 
represented a $1$-cocycle $a\in Z^1(Y;\Z_p)$ with
$a(z)=1$, where $z$ denotes the $1$-chain corresponding
to the circle in the $1$-skeleton. 
It follows from \cite[Theorem 2]{Porter} 
that the $p$-fold Massey product $\la v, \ldots , v \ra$ is defined 
and is equal to $- c_p(x)[b]$,  
where $c_p(x)$ is the coefficient of $x^p$ mod $p$ in
$(1 + x)^p = 1 + x^p$, and $b\in Z^2(Y;\Z_p)$ is the reduction 
mod $p$ of an integer $2$-cocycle $\beta$ determined by 
$d \alpha = p \beta$, where $\alpha \in C^1(Y;\Z)$ is an 
integer cochain with mod $p$ reduction equal to $a$.  
Note that in our case $c_p(x) = 1$, and from the equation
$d\alpha = p \beta$, it follows from the definition of
the Bockstein operator
that $[b] = \delta_p([a])$. So in $Y$, we have that 
$\la v, \ldots , v \ra = - \delta_p(v)$.

Now, since $Y$ is the $2$-skeleton of a $K(\Z_p,1)$, and since every 
element $u\in H^1(X;\Z_p)$ is equal to the pullback $f^*(v)$ of the class 
$v\in H^1(K(\Z_p,1);\Z_p)=H^1(Y;\Z_p)$ for a (cellular) map
$f\colon X \to  K(\Z_p,1)$, it follows from the naturality
of Massey products and the Bockstein operators that 
the $p$-fold Massey product $\la u, \ldots , u \ra$ is 
defined and contains the element $- \delta_p (u)\in H^2(X;\Z_p)$. 
\end{proof}

\begin{remark}
\label{rem:mod-p-bockstein}
Since the Bockstein $\delta_p(v)$ is nonzero in 
$H^2(K(\Z_p,1);\Z_p)$, the above proof provides 
examples where $p$-fold Massey products of  
elements $v\in H^1(K(\Z_p,1);\Z_p)$ do not vanish.
\end{remark}

\begin{remark}
\label{rem:mod-p-zeta-massey}
Let $u\in H^1(X;\Z_p)$ be a cohomology class represented 
by a cocycle $a\in Z^1(X;\Z_p)$. By equation \eqref{eq:da-cup}, 
we have that $b=- \sum_{i=1}^{p-1} \z_i(a) \cup \z_{p-i}(a)$ is a cocycle 
with cohomology class an element in the $p$-fold Massey product 
$\langle [a], \dots , [a] \rangle$.
\end{remark}

\subsection{Massey triple products with restricted indeterminacy}
\label{subsec:restricted}

In this section we define the restricted Massey triple
product, $\langle u_1, u_1, u_2 \rangle_r$.
Restricted Massey triple products are
invariants of $1$-equivalence of binomial $\cup_1$-dgas
where the coefficient ring $R$ is either $\Z$ or $\Z_p$, with
$p$ a prime and $p \ge 3$. In section \ref{subsec:types} we 
will use these invariants to give an example of a family 
of $2$-dimensional CW-complexes that have isomorphic 
cohomology rings yet are not $1$-equivalent.

Let $A$ be a binomial $\cup_1$-dga over a ring $R$ as above.  
Given elements $u_1, u_2 \in H^1(A)$ with $u_1 \cup u_2 = 0$, 
let $a_i$ be cocycle representatives for $u_i$, and let $a_{12}$ 
be a $1$-cochain such that $da_{12} = a_1 \cup a_2$.
By Lemma \ref{lem:good}, the cochain 
\begin{equation}
\label{eq:gamma-c}
\gamma(a_1, a_2, a_{1,2})\coloneqq a_1\cup a_{1,2} - \z_2(a_1)\cup a_2
\end{equation}
is a $2$-cocycle. 

\begin{definition}
\label{lem:restricted-mtp}
With notation as above, we define the \emph{restricted Massey 
triple product}\/ $\langle u_1, u_1, u_2 \rangle_r$ to be the 
subset of $H^2(A)$ consisting of all elements of the form
$[\gamma(a_1, a_2, a_{1,2})]$. 
The cochains $a_1, a_2, a_{1,2}$ are called a 
\emph{defining system}\/ for $\langle u_1,u_1, u_2\rangle_r$.
\end{definition}

The next lemma computes the indeterminacy of these restricted 
Massey triple products.  We restrict our attention to the case when 
$A=C^*(X;R)$ is the cochain algebra of a $\Delta$-complex $X$, 
with coefficients in the ring $R=\Z$ or $R=\Z_p$ with $p\ge 3$. 
We sketch a proof in this case, following a suggestion 
by the referee. The indeterminacy of restricted Massey products
will covered in more generality in \cite{Porter-Suciu-2021-3}.

\begin{lemma}
\label{lem:ind-res}
The indeterminacy of $\langle u_1, u_1, u_2 \rangle_r$ is 
the set $u_1 \cup H^1(A)$.
\end{lemma}

\begin{proof}  
It suffices to show that the subset of 
$\langle u_1, u_1, u_2\rangle_r$ obtained from any fixed choice
of cocycles representing $u_1$ and $u_2$ is the entire set
$\langle u_1, u_1, u_2\rangle_r$.  Let $a_1, a_2, a_{1,2}$ 
be a defining system for $\langle u_1,u_1, u_2\rangle_r$. 
If one changes $a_2$ to $\tilde{a}_2= a_2 + db$ for some 
$b \in C^0(X)$ and $a_{1,2}$ to 
$\tilde{a}_{1.2} = a_{1,2} -a_1\cup b$, then 
$\gamma(a_1, a_2, a_{1,2})$ changes by 
$d(\z_2(a_1) \cup b)$. On the other hand, 
if one changes $a_1$ to $\tilde{a}_1 = a_1 + db$
and $a_{1,2}$ to $\tilde{a}_{1,2} = a_{1,2} + b \cup a_2$,
then by using the formulas
\begin{align}
\z_2(a_1 + db) & = \z_2(a_1) + a_1 \cup_1 db + \z_2(db),\notag \\
 a_1 \cup_1 db & = a_1 \cup b- b \cup a_1 ,\\
d(c(b)) & = db \cup b - \z_2(db), \notag
\end{align}
where $c(b) \in C^0(X)$ is defined by
$c(b)(v) = \z_2(-b(v))$ for all vertices $v$ in $X$,
it follows that $\gamma(a_1, a_2, a_{1,2})$ changes by
$d(b \cup a_{1,2} + c(b)\cup a_2)$, and the proof is complete.
\end{proof}

\begin{remark}
\label{rem:additivity}
It follows straight from the definition that restricted 
Massey triple products satisfy the following additivity formula.
\begin{equation}
\label{eq:massey-add}
\langle u, u, w_1 \rangle_r + \langle u, u, w_2 \rangle_r 
\subseteq \langle u, u, w_1 + w_2 \rangle_r  \, .
\end{equation}
\end{remark}

\begin{lemma}
\label{lem:restricted-massey}
Let $\varphi\colon A\to B$ be a morphism between two binomial 
cup-one $R$-algebras, and let $\varphi^*\colon H^*(A)\to H^*(B)$ be the 
induced morphism in cohomology.  If $u_1, u_2 \in H^1(A)$ are such that 
$u_1 \cup u_2 = 0$, then $\varphi^*(\langle u_1, u_1, u_2\rangle_r) 
\subseteq \langle \varphi^*(u_1), \varphi^*(u_1), \varphi^*(u_2)\rangle_r$. 
Moreover, if $\varphi$ is a $1$-quasi-isomorphism, then the inclusion 
holds as equality.
\end{lemma}

\begin{proof}
The first claim follows from the naturality of cup-products and 
$\zeta$-maps. The second claim follows from the first one and 
Lemma \ref{lem:ind-res}.
\end{proof}

The next result shows that restricted Massey triple products are 
an invariant of $1$-equiva\-lence for binomial cup-one algebras.
Recall that for a graded algebra $H$, we denote by
$D^2H$ 
the image of the cup product map $H^1 \otimes H^1 \to H^2$.

\begin{proposition}
\label{prop:res-mas}
Let $A$ and $B$ be two $1$-equivalent binomial $\cup_1$-dgas over $R$. 
Then there is an isomorphism $\Phi\colon H^1(A) \oplus D^2(H(A)) \to 
H^1(B) \oplus D^2(H(B))$ which preserves degrees, commutes 
with cup products, and satisfies 
$ \Phi \left(\langle u_1, u_1, u_2\rangle_r\right)=
\langle \Phi(u_1),  \Phi(u_1),  \Phi(u_2)\rangle_r$ 
for all $u_1, u_2\in H^1(A)$ with $u_1 \cup u_2 = 0$.
\end{proposition}

\begin{proof}
Recall from \S\ref{subsec:dga} that a $1$-quasi-isomorphism 
$\varphi\colon A\to B$ 
induces an isomorphism $\varphi^*\colon H^{1}(A)\to H^{1}(B)$ 
and a monomorphim $\varphi^*\colon H^{2}(A)\to H^{2}(B)$. 
By naturality of cup-products, the latter map restricts to an 
isomorphism $\varphi^*\colon D^2(H(A))\to D^2(H(B))$. 
The claim now follows from the definition of $1$-equivalence 
and the previous lemma.
\end{proof}

Given an element $u \in H^1$, consider the following subset of $H^2$,
\begin{equation}
\label{eq:u-res}
\langle u \rangle_r = \bigcup_{\{w\in H^1\colon uw=0\}} \langle u, u, w \rangle_r \, .
\end{equation}

\begin{corollary}
\label{cor:ur}
Let $A$ and $B$ be two $1$-equivalent binomial $\cup_1$-dgas over $R$, 
and let $\Phi$ be the isomorphism from Proposition \ref{prop:res-mas}.
Then $\Phi(\langle u\rangle_r) =\langle \Phi(u)\rangle_r$ for all $u\in H^1$. 
\end{corollary}

\subsection{Distinguishing homotopy types}
\label{subsec:types}

Suppose now that $X$ and $Y$ are two finite $\Delta$-complexes 
such that their restricted triple Massey products cannot be matched 
as subsets of $H^2$ in the manner prescribed by 
Proposition \ref{prop:res-mas} or Corollary \ref{cor:ur}.  
Those results then imply that $X$ and $Y$ are not homotopy equivalent. 
We illustrate this method of distinguishing homotopy types 
with an example.

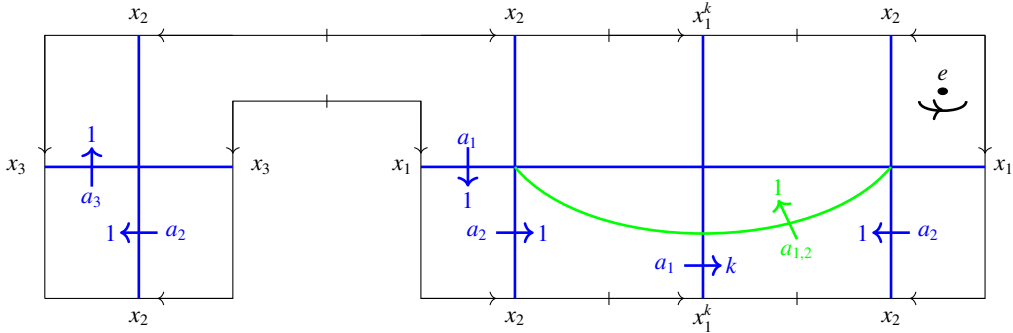
\begin{figure}
\begin{equation*}
\begin{gathered}
\begin{tikzpicture}[xscale=1.25, yscale=0.875]
\draw[black] (0,0) -- (2,0) -- (2,3) -- (4,3) -- (4,0) -- (10,0) 
											-- (10,4) -- (0,4) -- (0,0);
\draw [->] (0,4) -- (0,2.2);
\node at (-.3,2) {\scriptsize$x_3$};
\draw [->] (2,0) -- (1.2,0);
\node at (1,-.3) {\scriptsize$x_2$};
\node at (2.3,2) {\scriptsize$x_3$};
\draw [->] (2,3) -- (2,2.2);
\draw (3,2.9)--(3,3.1);
\draw (3,3.9)--(3,4.1);
\draw [->] (4,3) -- (4,2.2);
\node at (3.8,2) {\scriptsize$x_1$};
\node at (5,-.3) {\scriptsize$x_2$};
\draw [->] (4,0) -- (4.8,0);
\draw (6,-.1) -- (6,.1);
\node at (7,-.3) {\scriptsize$x_1^k$};
\draw [->] (6,0) -- (6.8,0);
\draw (8,-.1) -- (8,.1);
\node at (9,-.3) {\scriptsize$x_2$};
\draw [->] (10,0) -- (9.2,0);
\node at (5,4.3) {\scriptsize$x_2$};
\draw [->] (4,4) -- (4.8,4);
\draw (6,3.9) -- (6,4.1);
\node at (7,4.3) {\scriptsize$x_1^k$};
\draw [->] (6,4) -- (6.8,4);
\draw (8,3.9) -- (8,4.1);
\node at (9,4.3) {\scriptsize$x_2$};
\draw [->] (10,4) -- (9.2,4);
\draw [->] (10,4) -- (10,2.2);
\node at (10.2,2) {\scriptsize$x_1$};
\draw [->] (2,4) -- (1.2,4);
\node at (1,4.3) {\scriptsize$x_2$};
\draw [blue, line width=1pt] (0,2) -- (2,2);
\draw [blue, line width=1pt, ->] (.5,1.7) -- (.5,2.3);
\node [blue,thick] at (.5,1.5) {\scriptsize$a_3$};
\node [blue,thick] at (.5, 2.5) {\scriptsize $1$};
\draw [blue, line width=1pt] (4,2) -- (10,2);
\draw [blue, line width=1pt, ->] (4.5,2.3) -- (4.5,1.7);
\node [blue,thick] at (4.5,2.4) {\scriptsize$a_1$};
\node [blue,thick] at (4.5,1.5) {\scriptsize $1$};
\draw [blue, line width=1pt] (5,0) -- (5,4);
\draw [blue, line width=1pt, ->] (4.8,1) -- (5.2,1);
\node [blue,thick] at (4.6,1) {\scriptsize$a_2$};
\node [blue,thick] at (5.3,1) {\scriptsize $1$};

\draw [blue, line width=1pt] (7,0) -- (7,4);
\draw [blue, line width=1pt, ->] (6.8,.5) -- (7.2,.5);
\node [blue,thick] at (6.6,.5) {\scriptsize$a_1$};
\node [blue,thick] at (7.3,.5) {\scriptsize$k$};
\draw [blue, line width=1pt] (9,0) -- (9,4);
\draw [blue, line width=1pt, <-] (8.8,1) -- (9.2,1);
\node [blue,thick] at (9.4,1) {\scriptsize$a_2$};
\node [blue,thick] at (8.7,1) {\scriptsize $1$};
\draw [blue, line width=1pt] (1,0) -- (1,4);
\draw [blue, line width=1pt, <-] (.8,1) -- (1.2,1);
\node [blue,thick] at (.7,1) {\scriptsize $1$};
\node [blue,thick] at (1.4,1) {\scriptsize$a_2$};
\draw [green, line width=1pt] (5,2) [out=300,in=240] to (9,2);
\draw [green, line width=1pt, ->] (8,.9) -- (7.8,1.5);
\node [green, thick] at (8,.7) {\scriptsize{$a_{1,2}$}};
\node [green,thick] at (7.8, 1.7) {\scriptsize $1$};
\draw [line width=1pt] (9.3,3) [out=270,in=270] to (9.8,3);
\draw [line width=1pt,->] (9.5,2.87) -- (9.55,2.87);
\draw [fill] (9.55,3.15) circle [radius=0.05];;
\node at (9.55,3.4) {\scriptsize$e$};
\end{tikzpicture}
\end{gathered}
\end{equation*}
\caption{The relator $[x_2,x_3][x_1,x_2x_1^kx_2^{-1}]$}
\label{fig:distinguish}
\end{figure}
\begin{example}
\label{ex:mtp-r} 
Given a non-negative integer $k$, set $X_k$ equal to the presentation
$2$-complex for the group $G_k$ with generators $x_1, x_2, x_3$ 
and relator $[x_2,x_3][x_1,x_2 x_1^k x_2^{-1}]$ as 
pictured in Figure \ref{fig:distinguish}. Choose a subdivision 
of the $2$-cell in Figure \ref{fig:distinguish} into a $\Delta$-complex, $Y_k$,
whose cells are transverse to the line segments labelled $a_1, a_2, a_3$, 
and $a_{1,2}$. 

Then, as in Figures \ref{fig:torus} and \ref{fig:mtp-3}, the line segments 
with arrows and numbers determine elements in $C^1(Y_k;\Z)$.
The cochains determined by $a_1, a_2$, and $a_3$ are cocycles, $c_i$,
whose cohomology classes are denoted by $u_i = u_{i,k}$. 
The cocycles $c_i$ evaluated on the $1$-cycles $x_j$, are given by 
$c_i(x_j) = \delta_{i,j}$. The cochain $c_{1,2}$ determined by $a_{1,2}$ 
satisfies $dc_{1,2} = c_1 \cup c_2$. Using the computation of the cup 
product in Figure \ref{fig:torus} as a guide, it follows that
$u_1\cup u_2=u_1\cup u_3=0$, while $u_2\cup u_3$ generates 
$H^2(Y_k;\Z)=\Z$. Consequently, all the cohomology rings 
$H^*(X_k;\Z)\cong H^*(Y_k;\Z)$  are pairwise isomorphic. Furthermore, 
the triple Massey products  $\langle u_1,u_1,u_2 \rangle$  and 
$\langle u_1,u_1,u_3 \rangle$ have indeterminacy equal to 
the whole of $H^2(X_k;\Z)$; thus, these invariants do not 
distinguish the homotopy types of the spaces $X_k$, either. 

On the other hand, we can use restricted triple Massey products
to show that $X_k$ and $X_\ell$ are not of the same $2$-type 
if $k \ne \ell$.  To prove this claim, let  $g_k$  be 
a generator of $H^2(X_k;\Z) = \Z$.  We first show that 
\begin{equation}
\label{eq:u1k}
\langle u_{1,k} \rangle_r = \{ nkg_k \colon n \in \Z\}\, .
\end{equation}
Indeed, by Lemma \ref{lem:ind-res} 
the indeterminacy of $\langle u_{1,k}, u_{1,k}, w \rangle_r$ 
is zero for all $w\in H^1(X_k;\Z)$. 
Using Figure \ref{fig:distinguish} it can be shown that 
\begin{equation}
\label{eq:u1k-res}
\langle u_{1,k}, u_{1,k}, u_{2,k} \rangle_r = \pm k g_k
\end{equation}
(with sign depending on the choice of $g_k$) and 
$\langle u_{1,k}, u_{1,k}, u_{3,k} \rangle_r = 0$.
We have already seen that in general
$\langle u_{1,k}, u_{1,k}, u_{1.k}\rangle$ contains zero.
The claim now follows from formula \eqref{eq:massey-add}. 

Now suppose $X_k$ and $X_\ell$ have the same $2$-type, i.e., 
$\pi_1(X_k)\cong \pi_1(X_{\ell})$. Then, by Theorem \ref{thm:quasi-iso}, 
the cochain algebras of $X_k$ and $X_\ell$ are $1$-quasi-equivalent. 
It then follows from Corollary \ref{cor:ur} that there is an 
isomorphism 
\begin{equation}
\label{eq:iso-e}
\begin{tikzcd}[column sep=20pt]
\Phi\colon H^1(X_k;\Z) \oplus D^2H(X_k;\Z) \ar[r, "\cong"]& 
H^1(X_\ell;\Z) \oplus D^2H(X_\ell;\Z)
\end{tikzcd}
\end{equation}
such that $\Phi\left( \langle u_{1,k} \rangle_r \right) = 
\langle \Phi(u_{1,k} )\rangle_r$ where $D^2H(X_k;\Z)$
denotes the submodule of $H^2(X_k;\Z)$ whose elements
are sums of cup products of elements in $H^1(X_k;\Z)$.

Note that
the set of all elements $u \in H^1(X_k;\Z)$ with
$u w=0$ for all $w \in H^1(X_k;\Z)$ is the subspace of 
$H^1(X_k;\Z)$ generated by $u_{1,k}$.  
Since $\Phi$ commutes with cup products and is an
isomorphism on $H^2$, it follows that
$\Phi(u_{1,k}) = \pm u_{1,\ell}$. Hence,
$\Phi(\langle u_{1,k} \rangle_r)$ is either
$\langle u_{1,\ell} \rangle_r$ or
$\langle - u_{1,\ell} \rangle_r$.
Since $\z_2(-a)-\z_2(a) = a$, we see that
$\langle u_{1,\ell} \rangle_r = \langle - u_{1,\ell} \rangle_r$, 
and thus $\Phi\left( \langle u_{1,k} \rangle_r \right) = 
\langle u_{1,\ell} \rangle_r$. In view of \eqref{eq:u1k}, we 
conclude that $k=\ell$.
\end{example}

\begin{ack}
We wish to thank the referee for many valuable comments and suggestions 
that helped us to improve both the substance and the exposition of the paper.
\end{ack}

\end{document}